\documentclass[a4paper,12pt]{amsart}
\usepackage{amsmath}
\usepackage{breqn}
\usepackage{mathbbol}
\usepackage{extsizes}
\usepackage[utf8]{inputenc}
\usepackage{amssymb, latexsym}
\usepackage{xcolor}
\usepackage{hyperref}
\usepackage{ulem}
\usepackage{amsthm}
\usepackage{comment}
\usepackage{cancel}
\newcommand{\w}{\widehat}
\newcommand{\wt}{\widetilde}

\newtheorem{theorem}{Theorem}[section]
\newtheorem{corollary}{Corollary}[theorem]
\newtheorem{lemma}[theorem]{Lemma}
\newtheorem{proposition}[theorem]{Proposition}
\newtheorem{remark}{Remark}
\newtheorem{definition}{Definition}
\title[Conjugacy Classes of Automorphisms]{Conjugacy classes of Automorphisms of the unit ball in a complex Hilbert space}

\author[R. Aggarwal, K. Gongopadhyay and M. M. Mishra  ]{Rachna Aggarwal, Krishnendu Gongopadhyay and Mukund Madhav Mishra }
\address{Department of Mathematical Sciences, Indian Institute of Science Education and Research (IISER) Mohali, Knowledge City,
S.A.S. Nagar, Sector 81, P.O. Manauli 140306, India}
\email{rachna2389@gmail.com}
\address{Department of Mathematical Sciences, Indian Institute of Science Education and Research (IISER) Mohali, Knowledge City,
S.A.S. Nagar, Sector 81, P.O. Manauli 140306, India}
\email{krishnendu@iisermohali.ac.in}
\address{Department of Mathematics, Hans Raj College, University of Delhi, Delhi, India}
\email{mukund@hrc.du.ac.in; mukund.math@gmail.com}

\numberwithin{equation}{section}
\date{}
\begin{document}
\begin{abstract}
In this article, we consider the ball model of an infinite dimensional complex hyperbolic space, i.e. the open unit ball of a complex Hilbert space  centered at the origin equipped with the Carath\'eodory metric. We consider the group of holomorphic automorphisms of the ball and classify the conjugacy classes of automorphisms.We also compute the centralizers for elements in the group of automorphisms.

\end{abstract}

\maketitle
\section{Introduction}
The rank one symmetric spaces of non-compact type are given by real,  complex and the quaternionic hyperbolic spaces, and the Cayley plane. These spaces are of primary interest to mathematicians for their broad horizons over different branches of mathematical sciences.  For a uniform and gentle introduction to the first three of these spaces we refer to the article \cite{CG}.  From geometric perspective, the real hyperbolic spaces, especially in low dimensions, have seen much attention due to their importance in the theory of Kleinian groups, e.g.  \cite{JA}, \cite{eps}. The complex hyperbolic spaces also play an important role in mathematics due to their connection with lattices in Lie groups and deformation theory of discrete groups, cf.  \cite{WG}, \cite{parker1}, \cite{parker2}.  Recently quaternionic hyperbolic spaces have also seen many  investigations, e.g. \cite{kip}, \cite{kp}, \cite{kiv}.

In spite of many works on the finite dimensional hyperbolic spaces,  their infinite dimensional counterparts have not seen much attention until  Gromov suggested investigations of these infinite dimensional hyperbolic spaces from algebraic and geometrical point of views, cf.
 \cite[pp. 121]{gromov}.  In the last two decades, the literature has seen many contributions concerning the geometry of the infinite dimensional real hyperbolic space,  e.g.  \cite{bim},  \cite{dsu},  \cite{np}.  However, the geometry of the infinite dimensional complex hyperbolic space is yet to see comparable attentions.

We attempt investigation of the infinite dimensional complex hyperbolic space by studying conjugacy classes of the isometries and their centralizers in this article. We recall here that in the finite dimensional set up,  dynamical properties of the isometries are intimately related to the fixed point classification of the isometries, and this is related to the conjugacy classes.  In the finite dimensional set up, the conjugacy classification can be described completely using conjugacy invariants like the coefficients of the characteristic polynomials, cf. \cite{KG}, \cite{g}, \cite{gp}, \cite{gpp}.    However, the situation in infinite dimension is not simple not only due to lack of conjugacy invariants, but also because of complexity in spectra.  The starting point of our work would be to use the `ball model' of the complex hyperbolic space that was introduced by Franzoni and Vesentini in \cite[Chapter VI]{FV}.  They have also given a linear representation of the isometries of the ball model that we will use extensively. Using similar ideas, we also introduce a `Siegel domain model' of the infinite dimensional complex hyperbolic space. We use these different models to classify the conjugacy classes of the isometries, and obtain discriptions for the centralizers. 

In \cite{ku}, Kulkarni proposed to use the centralizers as an internal ingredient of the group of automorphisms in a geometry to classify the automorphisms into disjoint classes which roughly reflect the `dynamical types'.  This indicates the importance of the conjugacy classes of centralizers that might be implicit in geometry.  The  centralizers in finite dimensional set up have been investigated in \cite{g}, \cite{KG},   \cite{JC}, \cite{WCKG}. In this paper, we describe the centralizers  for the elliptic and the hyperbolic elements. The conjugacy classes of centralizers for such elements may be derived from this description. We also provide descriptions of the centralizers for `translations'. As corollaries, we observe when certain isometries commute to each other. 

\medskip Before stating the main results we fix some notations. 
 Let $K$ be a Hilbert space and $\left<\,,\,\right>$  denote the  inner product on $K$.   For $x \in K$,  $\left<x\right>$ will denote the linear space generated by $x$. For a subspace $M \subseteq K$, ${M}^{\perp}$ denotes the orthogonal complement of $M$ with respect to the inner product on $K$. The space of all bounded linear operators on $K$ will be denoted by $B(K)$  and for $T \in B(K)$, $\sigma(T)$ denotes the spectrum of $T$. For a set $X \subseteq K$, $\overline{X}$ denotes the closure of $X$ with respect to the norm on $K$ and $\partial{X}$, its boundary. The unit circle in $\mathbb{C}$ is denoted by $S^1$. For any two operators $T$ and $S$ in $B(K)$, $T \longleftrightarrow S$ would mean $T$ and $S$ commute with each other, i.e. $TS=ST$.   An operator $T \in B(K)$ is called unitary if $T^*=T^{-1}$ where $T^*$ is the Hilbert-adjoint of $T$  and $T$ is called  normal if $TT^*=T^*T$. It is called  non-unitary normal if it is a normal operator which is not unitary.  The group of unitary operators in $B(K)$ will be denoted by $\mathcal{U}(K)$.

 \medskip
 Let  $H$ be  a complex Hilbert space. Suppose $B$ is the open unit ball in $H$ centered at the origin equipped with the Carathe\'odory metric. The group of bi-holomorphic automorphisms of $B$ is denoted by $Aut(B)$. The elements of $Aut(B)$ are isometries for the Carath\'eodory metric. An element in $Aut(B)$ is called  {\it elliptic} if it has a fixed point inside $B$, {\it hyperbolic}   (resp. {\it parabolic}) if it is not elliptic and has exactly two  (resp. one) fixed points in  $\overline{B}$ lying on $\partial{B}$. In this article we mainly study the conjugacy classes and centralizers of elements in the group $Aut(B)$.\\ Let $G$ be the group of all bijective bounded linear transformations on $H \oplus \, \mathbb{C}$ which are isometries for the following sesquilinear form defined on $H \oplus \, \mathbb{C}$.
\begin{eqnarray}\label{eq}
\mathcal{A}\left((x, z),\,(y,w)\right)=\left<x,y\right> -z \overline{w}, \,\,\,\,\,\,\,\, (x,z),\, (y,w) \in H \oplus \, \mathbb{C}.
\end{eqnarray}
Let $\mathcal{Q}$ be the quadratic form determined by $\mathcal{A}$.
\begin{proposition}\cite[Proposition 1]{MR} \label{p}  A general element of $G$ is of the form $e^{i \theta} \left[ {\begin{array}{cc}
   UA & U(\xi) \\
  \left<\boldsymbol{\cdot},\xi\right>  & a \\
  \end{array} } \right] \,\text{where}\,\theta \in \mathbb{R}$,  $\xi \in H$, $a=\sqrt{1+\|\xi\|^2}$, $U \in \mathcal{U}(H)$ and A is a positive operator on $H$ defined by $A=I$ on $\left<\xi\right>^{\perp}$ and $A(\xi)=a\xi$.
 \end{proposition}
  For  a closed subspace $M \subseteq H \oplus \mathbb{C}$, let  $G_{M}$ denote the group of all bijective bounded linear mappings on $M$ leaving the form $\mathcal{A}\restriction_{M}$ invariant.
  For a subspace $M \subseteq H \oplus \, \mathbb{C}$, orthogonal complement of $M$ for the Hermitian form $\mathcal{A}$ defined in (\ref{eq}) will be denoted by $M^{\dagger}$. A vector $\mathbb{x} \in H \oplus \, \mathbb{C}$ is called  time-like  if $\mathcal{Q}(\mathbb{x})<0$ and  light-like  if $\mathcal{Q}(\mathbb{x})=0$.  Isometries in $Aut(B)$ are projective transformations induced from elements in $G$. We would use the word isometry to denote the elements of $Aut(B)$ and $G$ simultaneously. 
  
  An isometry in $G$ is called elliptic, hyperbolic or parabolic if a corresponding element in $Aut(B)$ is of the  respective type. An eigenvalue of an isometry $T \in G$ is called time-like if its corresponding eigenspace has a  time-like vector. An elliptic isometry in $G$ is called boundary elliptic if the geometric multiplicity of its time-like eigenvalue is atleast two, otherwise it is called regular elliptic.  The words `normal isometry' and `non-unitary normal isometry' would mean isometries in $G$ which are normal and non-unitary normal as Hilbert space linear operators. The center of the group $G$ will be denoted by $Z_G$. The centralizer of an isometry $T \in G$ will be denoted by $Z(G)$.

  We will study parabolic isometries on a Siegel domain model which is a counterpart of the half plane model in the classical case. Let $e \in \partial{B}$. We write a general element $x \in H$ as $x=\left<x,e\right>e+x'$ where $x' \in \left<e\right>^{\perp}$.  We consider the following sesquilinear form.
  \begin{eqnarray}\label{se}
  \widehat{\mathcal{A}}\left((x,z),(y,w)\right)=-z\overline{\left<y,e\right>}-\left<x,e\right>\overline{w}+\left<x',y'\right>.
  \end{eqnarray}
  The projective space of its negative vectors gives rise to the  Siegel domain. It is defined as
  \[\Sigma=\left\{x \in H \,\,:\,\, \text{Re}\,\left<x,e\right> > \frac{1}{2}\|x'\|^2\right\}.\]
  Hermitian forms (\ref{eq}) and (\ref{se}) are equivalent via a unitary operator $D=\left[\begin{array}{ccc}
{}\dfrac{1}{\sqrt{2}} & {}0 & {}-\dfrac{1}{\sqrt{2}}\\
{}0 & {}I & {}0 \\
 {}\dfrac{1}{\sqrt{2}} & {}0 & {}\dfrac{1}{\sqrt{2}}
 \end{array}\right].$ The projective map induced by $D^{-1}$ is a Cayley map which is a biholomorphism between the unit ball $B$ and the Siegel domain $\Sigma$.
  Let the image of $e$ under Cayley map  be denoted by  $\infty$ which is the point at infinity in the completion of $\overline{\Sigma} \setminus \infty$.  Let $\widehat{G}$ denote the group of all bijective bounded linear mappings on $H \oplus \mathbb{C}$ leaving $\widehat{\mathcal{A}}$ invariant.  Similar to the $Aut(B)$ case, the isometries of $Aut(\Sigma)$ are the projective maps induced by elements of $\widehat{G}$.  \\
  For a subspace $M$ of $H \oplus \mathbb{C}$, we will use the same notation $M^{\dagger}$ to denote the orthogonal complement of $M$ with respect to $\widehat{\mathcal{A}}$.
  Let $\widehat{G}_\infty$ denote the stabilizer group of infinity in $\widehat{G}$.  A general element of $\widehat{G}_\infty$ is of the form $\left[\begin{array}{ccc}
  \lambda    & \left<\boldsymbol{\cdot},U^{-1}(a')\right> &   \mu s\\
    0  & U & \mu a'\\
    0 & 0 & \mu
 \end{array}\right]$ where  $\lambda \in \mathbb{C}$ is such that $\lambda\,\overline{\mu}=1$, $U \in \mathcal{U}(\left<e\right>^{\perp})$,  $s \in \mathbb{C}$ and $a' \in \left<e\right>^{\perp}$ satisfy $\text{Re}\,s=\dfrac{1}{2}\|a'\|^2$. The form of a  parabolic isometry in $\widehat{G}_\infty$ having singleton spectrum is given by $\lambda \left[\begin{array}{ccc}
  1    & \left<\boldsymbol{\cdot},a'\right> &   s\\
    0  & I & a'\\
    0 & 0 & 1
 \end{array}\right]$ which is denoted by $(\lambda,a',s)$ such that $s \neq 0$. Projective maps induced by such isometries are  Heisenberg translations with $a'$ as component in the horizontal direction. Thus a Heisenberg translation is called vertical if $a'=0$   otherwise it is called non-vertical.

\subsection{Main Results} 
  Following are the main results of this article.  Theorem \ref{e}  describes an elliptic isometry and its conjugacy class, Theorem \ref{t1}  investigates a hyperbolic isometry along with its conjugacy class and Theorem \ref{o} explores the conjugacy of parabolic isometries having singleton spectra. Centralizers  of  the isometries  are  given  in Theorems \ref{p1}, \ref{p3} and \ref{t3} respectively.
  \begin{theorem}[Elliptic isometry]\label{e}
 Let $T  \in G$ be an elliptic isometry. Then
 \begin{enumerate}
 \item $T$ is unitary upto conjugacy.
 \item $\sigma(T) \subseteq S^1$.
 \item $T$  has a  time-like eigenvector and vice versa.
 \item   If $(x,1)$ is a  time-like eigenvector of $T$ then $T=T_1 \oplus \, T_2$  where $T_1=T\restriction_{\left<(x,1)\right>}$  and  $T_2=T\restriction_{\left<{(x,1)}\right>  ^{\dagger}}$. Moreover $T_2 \in \mathcal{U}({\left<{(x,1)}\right>  ^{\dagger}})$.
 \end{enumerate}
 \end{theorem}
 \begin{theorem}[Hyperbolic isometry]\label{t1}
 Let $T \in G$ be a hyperbolic isometry. Then
\begin{enumerate}
\item $T$ is non-unitary normal upto conjugacy.
\item Spectrum of T has two eigenvalues of the form $re^{i \theta}$ and $r^{-1}e^{i \theta}$ where $r>0$, $r \neq 1$. Eigenspaces with respect to these eigenvalues are one dimensional, each generated by a light-like eigenvector. Rest of the spectral values lie on $S^1$.
\item $T=T_1 \oplus  T_2$  where $T_1=T\restriction_{M}$, $M=\text{span}\,\left\{(y_1,1), (y_2,1)\right\}$,  $(y_1,1)$ and $(y_2,1)$ are light-like eigenvectors of $T$ and $T_2=T\restriction_{{M}^{\dagger}} $. Moreover $T_2 \in \mathcal{U}(M^{\dagger})$.
  \end{enumerate}
 \end{theorem}
\begin{theorem}[Parabolic isometry]\label{o}
 Let $\widehat{S} \in \widehat{G}$ be a parabolic isometry  such that $\sigma(\widehat{S})=\{\lambda\}$. Then
 \begin{enumerate}
 \item Upto conjugacy, $\widehat{S}$ is  a Heisenberg translation.
      \item $\widehat{S}=\widehat{S}_1 \oplus \widehat{S}_2$ where $\widehat{S}_1=\widehat{S}\restriction_K$ and $\widehat{S}_2=\widehat{S}\restriction_{K^{\dagger}}$. Here $K$ is a two or three dimensional non-degenerate subspace of $H \oplus \mathbb{C}$ containing a light-like eigenvector and $\widehat{S}_2=\lambda\,I$. The $ker\,(\widehat{S_1}-\lambda\,I)$ is generated by the light-like eigenvector and minimal polynomial of $\widehat{S}_1$ is $(x-\lambda)^2$ or $(x-\lambda)^3$.
      \item All the parabolic isometries in $\widehat{G}$ having same singleton spectra and degree of restricted minimal polynomial as $2$ (resp.  $3$)  get dispersed into two conjugacy classes (resp. fall in a single conjugacy class).
      \end{enumerate}
  \end{theorem}
 \begin{theorem}[Centralizer of an elliptic isometry]\label{p1}
 Let $T$ be an elliptic isometry having $M$ as its time-like eigenspace. Then
  \[Z(T)=G_{M} \times Z(T\restriction_{M^{\dagger}}) \]
 where $Z(T\restriction_{M^{\dagger}}) \subseteq \mathcal{U}(M^{\dagger})$.
\end{theorem}
\begin{theorem}[Centralizer of a hyperbolic isometry]\label{p3}
 Let $T$ be a hyperbolic isometry and $M$ be the two-dimensional subspace of $H \oplus \mathbb{C}$  generated by two light-like eigenvectors of $T$. Then
 \[Z(T)=Z(T\restriction_{M}) \times Z(T\restriction_{M^{\dagger}})\]
    where $Z(T\restriction_{M})$ gets identified with $S^1 \times \mathbb{R}$ and $Z(T\restriction_{M^{\dagger}}) \subseteq \mathcal{U}(M^{\dagger})$.
\end{theorem}

The following two results also follow from the proofs of the above theorems. 
 \begin{corollary}\label{c5}
  Hyperbolic isometries either commute with boundary elliptic isometries or  hyperbolic isometries.
  \end{corollary}

  \begin{corollary}
  Two commuting hyperbolic isometries have the same fixed points.
  \end{corollary}

\begin{theorem}[Centralizer of  Heisenberg translation]\label{t3}
Let $\widehat{T}=(\lambda,a',s)$ be a Heisenberg translation. Then
   \begin{enumerate}
   \item For $a'=0$, centralizer of a vertical translation $\widehat{T}$ consists of all non-hyperbolic isometries in $\widehat{G}_\infty.$
 \item For a  non-vertical translation $\widehat{T}$,
 \[Z(\widehat{T})=\left\{ \left[\begin{array}{ccc}
  \lambda'    & \left<\boldsymbol{\cdot},U^{-1}(b')\right> &   \lambda' t\\\\
    0  & U & \lambda' b'\\\\
    0 & 0 & \lambda'\\
 \end{array}\right] \in \widehat{G}_\infty\,\,:\,\,U(a')=\lambda'\,a',\,\,\left<b',a'\right> \in \mathbb{R}\right\}.\]

   \end{enumerate}
 \end{theorem}
\textit{Structure of the article.} In section 2, we describe the group $Aut(B)$ along with its linear representation. In section 3, we study some significant sub-classes of $G$, provide a class of examples of each of elliptic, hyperbolic and parabolic isometries  and give infinite dimensional versions of some general results in finite dimension. Section 4 investigates conjugacy classes of elliptic and hyperbolic isometries. Section 5 studies parabolic isometries along with the discussion of conjugacy within a subclass consisting of singleton spectra. Section 6 investigates centralizers of elliptic, hyperbolic, and parabolic isometries having singleton spectra.

\subsection*{Acknowledgements}  The authors would like to thank John R. Parker for many comments and suggestions that has improved this work. The authors also acknowledge partial support from the TARE project grant TAR/2019/000379 and the SERB core research grant CRG/2022/003680. 
\section{Holomorphic automorphisms of $B$}
We now present the following set up as described in \cite {FV} which forms the basis of our study.\\

  \begin{lemma}\cite[Proposition III.2.2.]{FV} \label{lemma 1}
           An isometry in $Aut(B)$ fixing the origin is the restriction of a linear bijection on $H$ to $B$.
          \end{lemma}
          \begin{lemma}\cite[Proposition VI.1.2.]{FV} \label{lemma 2}
          A linear map $U \in B(H)$ belongs to $Aut(B)$ if and only if $U$ is unitary.
          \end{lemma}
\begin{theorem} \cite[Theorem VI.1.3.]{FV}  A general element of $Aut(B)$ is of the form $F=U\circ f_{b}$,  $b\in B$, where $U \in \mathcal{U}(H)$ and
$x \mapsto T_{b}\left( \dfrac{x-{b}}{1-\left<x,b\right>} \right)$ which defines $f_b$ is a holomorphic automorphism of $B$.
 The map $T_{b} : H\longrightarrow H$ is a linear map expressed as
\[T_{b}(x)=\dfrac{\left<x,b \right>}{1+\sqrt{1-\|b\|^2}}b+\sqrt{1-\|b\|^2}\,\,x.\]
\end{theorem}
Here $f_b(b)=0$ and ${f_b}^{-1}=f_{-b}$.

$Aut(B)$  acts transitively on B  \cite[Proposition VI.1.5.]{FV}.\\

\textit{We will be simultaneously using the same notation for inner products on Hilbert spaces $H$ and $H \oplus \, \mathbb{C}$}. \\

A linear representation of the isometries of $Aut(B)$ is described as follows  \cite[Ch. VI, sec. 3]{FV}. Let $\mathcal{A}$ denote the Hermitian form on $ H \oplus \, \mathbb{C}$ defined by
\begin{eqnarray}\label{eq:8}
\hspace{1cm} \mathcal{A}\left((x, z),\,(y,w)\right)=\left<x,y\right> -z \overline{w}, \,\,\,\,\,\,\,\, (x,z),\, (y,w) \in H \oplus \, \mathbb{C}\label{eq:6}
\end{eqnarray}
and $\mathcal{Q}$ be the  quadratic form determined by $\mathcal{A}$. Let \(G\) be the group  of all bijective bounded linear transformations on $H \oplus \, \mathbb{C}$ leaving $\mathcal{A}$ invariant.\\
General form of an element of $G$ is $ T=\left[ {\begin{array}{cc}
   A & \xi \\
  \left<\boldsymbol{\cdot},\dfrac{A^*(\xi)}{a}\right>  & a \\
  \end{array} } \right]
$, cf. \cite[Lemma ~VI.3.1]{FV},  where  $A \in B(H)$ is bijective, $\xi \in H$,  $a \in \mathbb{C}$, $|a|^{2}=1+\|\xi\|^{2}$, and 
\begin{eqnarray*}
A^{*}A =I+\dfrac{1}{|a|^{2}}\left<\boldsymbol{\cdot},{A^{*}(\xi)}\right>A^{*}(\xi).
\end{eqnarray*}

The center $Z_G$ of $G$ is $\{e^{i \theta}{I}, \,\,\,\, \theta \in \mathbb{R}\}$, see    \cite[Lemma VI.3.4]{FV}. The following theorem tells that the group $G$ acts on $B$  surjectively, see \cite{MR} for the proof.
\begin{theorem} \cite[Theorem VI.3.5]{FV} \label{a} The map $\phi : G \longrightarrow \text{Aut(B)}$ defined by  $T \mapsto \widetilde{T}$ is an onto homomorphism where \\
$T=\left[ {\begin{array}{cc}
   A & \xi \\
  \left<\boldsymbol{\cdot},\dfrac{A^*(\xi)}{a}\right>  & a \\
  \end{array} } \right]$ and $\widetilde{T}(x)= \dfrac{A(x)+\xi}{\left<x,\dfrac{A^*(\xi)}{a}\right>+a}$, $x \in B$.
\end{theorem}
Since  \({ker}(\phi )=Z_G\), the map \(\widetilde{\phi }:G/{Z_G} \longrightarrow Aut(B)\) is an
isomorphism.\\
\textit{Often we shall call $T$ as linearization of $\widetilde{T}$}.
\section{Some special sub-classes of $G$}
 Proposition \ref{p} gives a simplified form of elements of $G$.
 Notice that in the proposition
 \begin{eqnarray}\label{eq:7}
 a=\sqrt{1+\|\xi\|^2},
 \end{eqnarray}
 thus  $a \geq 1$.
   The Hilbert-adjoint operator for $T=e^{i\theta}\left[ {\begin{array}{cc}
   UA & U(\xi) \\
  \left<\boldsymbol{\cdot},\xi\right>  & a \\
  \end{array} } \right] \in G$  is
   $T^* =e^{-i\theta}\left[ {\begin{array}{cc}
   (UA)^* & \xi \\
  \left<\boldsymbol{\cdot},U(\xi)\right>  & a \\
  \end{array} } \right]$.
  Observe that $\mathcal{A}\left((x, z),(y, w)\right)=\left<A'(x,z), (y, w)\right>$ where
  \begin{eqnarray}\label{A'}
   A'=\left[ {\begin{array}{cc}
   I & 0 \\
   0 & -1 \\
  \end{array} } \right].
  \end{eqnarray}
  For any two elements $\mathbb{x},\,\mathbb{y} \in H \oplus \, \mathbb{C}$ and $T \in G$,  $\left<A'\left(T(\mathbb{x})\right), T(\mathbb{y})\right>=\left<A'(\mathbb{x}), \mathbb{y}\right>$, i.e. $T^*A'T=A'$, i.e.  $A'T^*A'={T}^{-1}$. This gives    $T^{-1}=\left[ {\begin{array}{cc}
   (UA)^* & -\xi \\
   -\left<\boldsymbol{\cdot},U(\xi)\right> & a \\
  \end{array} } \right]$.
  \begin{remark}
  Every isometry of $G$  in Proposition \ref{p} is expressible as the product of a unitary and a self-adjoint element of $G$, i.e.\\ $\left[ {\begin{array}{cc}
   UA & U(\xi) \\
  \left<\boldsymbol{\cdot},\xi\right>  & a \\
  \end{array} } \right]=\left[ {\begin{array}{cc}
   U & 0 \\
  0  & 1 \\
  \end{array} } \right]\left[ {\begin{array}{cc}
   A & \xi \\
  \left<\boldsymbol{\cdot},\xi\right>  & a \\
  \end{array} } \right].$
  \end{remark}
Every isometry in $Aut(B)$ is the restriction of a holomorphic map of some open ball containing $B$ into $H$ (see \cite[Proposition VI.1.4]{FV}). It  can be observed from computations in \cite[eq. (VI.1.6)-(VI.1.8)]{FV}  that  isometries of $Aut(B)$ leave $\partial{B}$ invariant and are bijective on it. The following theorem is due to Hayden and Suffridge which gives a fixed point classification for elements in $Aut(B)$.
 \begin{theorem} \cite[Hayden and Suffridge]{HS}
 Every element $g \in Aut(B)$ has a fixed point in $\overline{B}$. If $g \in \text{Aut(B)}$ has no fixed point in \text{B}, then the fixed point set in $\overline{\text{B}}$ consists of atmost two points.
 \end{theorem}
 We call an isometry in $Aut(B)$ \textbf{\textit{elliptic}} if has a fixed point in $B$, \textbf{\textit{hyperbolic}}  (resp. \textbf{\textit{parabolic}}) if it is not elliptic and
has exactly two (resp.  one) fixed points  on  \(\partial {B}\).

We say an isometry of $G$ is elliptic (resp. hyperbolic or parabolic) if it is in the  pre-image of an
elliptic (resp. hyperbolic or parabolic) isometry of $Aut(B)$ under the homomorphism \(\phi \)  in Theorem \ref{a}.

We call a vector $(x,z) \in H \oplus \, \mathbb{C}$  \textit{ \textbf{time-like}} if $\mathcal{Q}(x,z)<0$, \textit{ \textbf{light-like}} if $\mathcal{Q}(x,z)=0$ and \textit{ \textbf{space-like}} if $\mathcal{Q}(x,z)>0$. A linear subspace $W \subseteq H \oplus \, \mathbb{C}$ will be called \textbf{\textit{time-like space}} if it contains a time-like vector.
\begin{remark} \label{R1}
For $x \in \overline{B}$, $x$ is a fixed point for an isometry  $\dfrac{UA(\boldsymbol{\cdot})+U(\xi)}{\left<\boldsymbol{\cdot}, \xi \right>+a } \linebreak \in \text{Aut(B)}$ if and only if $(x,1)$ is an eigenvector  for a corresponding isometry in $G$.
\end{remark}


  \begin{lemma}[\cite{MR}, Proposition 3]\label{n}
 Let \(T=\left[ {\begin{array}{cc} UA &{} r\xi  \\ \left<\boldsymbol{\cdot} ,\xi \right> &{} a\\ \end{array}}\right] \in G\), $\xi \neq 0$,
\(|r|=1\). Then $T=T_1 \oplus \, T_2$ where $T_1=T\restriction_{  \left<\xi\right> \oplus \, \mathbb{C}}$ and $T_2=T\restriction_{(\left<\xi\right> \oplus \mathbb{C})^{\perp}} \equiv U\restriction_{\left<\xi\right>^{\perp}}$. Also
\begin{enumerate}
\item \({\sigma (T)}=\{\lambda_1,\,\lambda_2\} \cup {\sigma (U\restriction_{\left<\xi \right>^{\perp }})}\)
where
\(\lambda_1, \,\lambda_2 \,\,\text{are}\,\,\dfrac{a(r+1)\pm \sqrt{a^2 (r+1)^2-4r}}{2}\). The
eigenspaces corresponding to the eigenvalues \(\lambda_1\) and \(\lambda_2\) are generated by the
eigenvectors \(\left( k_1\xi,1 \right) \) and \((k_2 \xi,1)\) respectively where \(k_1,\,k_2\) are \(\dfrac{a(r-1)\pm
\sqrt{a^2(r+1)^2-4r}}{2 \Vert \xi \Vert ^2}\).
\item \(|\lambda_1|=\dfrac{1}{|\lambda_2|}\) and \(\Vert k_1\xi \Vert =\dfrac{1}{\Vert k_2\xi \Vert }\).
\end{enumerate}
 \end{lemma}
 The isometries considered in the above lemma constitute of a subclass of $G$ whose elements decompose $H \oplus \, \mathbb{C}$ orthogonally into a two dimensional subspace  containing $\mathbb{C}$ and its orthogonal complement (see \cite[Proposition 2]{MR}). \\
 The next result provides a class of  examples of each of  elliptic, hyperbolic and parabolic isometries.

 \begin{proposition} \label{c}
  Let  $T=\left[{\begin{array}{cc}
  UA & r\xi\\
  \left<\boldsymbol{\cdot},\xi\right> & a\\
  \end{array}}\right] \in G$, $\xi \neq 0$, $|r|=1$ be such that $T=T_1 \oplus T_2$ (cf. Lemma \ref{n}). Then $T$ is elliptic for $r=-1$, parabolic for $r=\left(\dfrac{2}{a^2}-1 \right)\pm i \left(\dfrac{2\|\xi\|}{a^2}\right)$ and hyperbolic for rest of the values of $r$ on $S^1$.
  \end{proposition}
\begin{proof}
In view of Lemma \ref{n}, we perform the following analysis. Let $r=-1$. This gives that $k_1=\dfrac{-2a+2}{2\|\xi\|^2}=\dfrac{1-a}{\|\xi\|^2}$. So ${k_1}^2=\dfrac{(1-a)^2}{\|\xi\|^4}$ and  $\|k_1 \xi\|^2={k_1}^2\|\xi\|^2=\dfrac{(1-a)^2}{\|\xi\|^2}$. Observe that by (\ref{eq:7}), $(1-a)^2-\|\xi\|^2=1+a^2-2a+1-a^2=2-2a<0$ as $\xi \neq 0$. This implies that   $(1-a)^2<\|\xi\|^2$, i.e. $\dfrac{(1-a)^2}{\|\xi\|^2} <1$. Hence  $T$ is elliptic by Remark \ref{R1}.

We will see that $r=\left(\dfrac{2}{a^2}-1\right) \pm i \left(\dfrac{2\|\xi\|}{a^2}\right)$ if and only if $\lambda_1=\lambda_2$.  From the value  of $\lambda_i$, we have
$\lambda_1=\lambda_2$ if and only if  $a^2(r+1)^2=4r$. 
This implies $\text{Re r}=\dfrac{2}{a^2}-1$.
 
 As $|r|^2=(\text{Re}\,r)^2+(\text{Im}\,r)^2=1$, $(\text{Im}\,r)^2=1-\left(\dfrac{2}{a^2}-1\right)^2=\dfrac{4}{a^2}-\dfrac{4}{a^4}=\dfrac{4(a^2-1)}{a^4}=\dfrac{4\|\xi\|^2}{a^4}$ (by (\ref{eq:7})) which gives that  $r=\left(\dfrac{2}{a^2}-1\right)\pm i\left(\dfrac{2\|\xi\|}{a^2}\right)$. Thus $\lambda_1=\lambda_2$ if and only if $r=\left(\dfrac{2}{a^2}-1\right)\pm i\left(\dfrac{2\|\xi\|}{a^2}\right)$. Also $k_i=\dfrac{a(r-1)}{2\|\xi\|^2}$. This implies
 \begin{align*}
  \|k_{i}\xi\|^2 &=|k_{i}|^2\|\xi\|^2\\
  &=\dfrac{a^2|r-1|^2}{4\|\xi\|^2}\\
  &=\dfrac{a^2(r-1)(\overline{r}-1)}{4\|\xi\|^2}\\
  &=\dfrac{a^2(1-r-\overline{r}+1)}{4\|\xi\|^2}\\
  &=\dfrac{a^2(2-2\,\text{Re}\,r)}{4\|\xi\|^2}\\
  &=\dfrac{a^2(1-\,\text{Re}\,r)}{2\|\xi\|^2}\\
  &=1
  \end{align*}
   as $1-\,\text{Re}\,r=1-\dfrac{2}{a^2}+1=2-\dfrac{2}{a^2}=\dfrac{2\|\xi\|^2}{a^2}$.
So for the case of identical eigenvalues of $T_1$, it has a fixed point on $\partial{B}$ in view of Remark \ref{R1}.  We will see that $T$ cannot have a fixed point inside the ball and has exactly one fixed point on $\partial{B}$. There will be two possibilities. Either $\sigma(T_1) \cap \sigma(T_2)=\emptyset$ or $\sigma(T_1) \cap \sigma(T_2)\neq \emptyset$. In the former case, the eigenvectors of $T$ will  come strictly either from $\left<\xi\right> \oplus \mathbb{C}$ or from $\left<\xi\right>^{\perp}$. There is exactly one eigenvector of $T_1$ of the form $(x,1)$ for which $x=k_i\xi \in \partial{B}$ and eigenvectors of $T_2$ are of the form $(y,0)$, if any, for $y \in \left<\xi\right>^{\perp}$. Thus in this case, $T$ does not have a fixed point inside the ball and cannot have more than one fixed point on the boundary in view of Remark \ref{R1}. In the latter case, let $(y,0)$ be an eigenvector of $T_2$ corresponding to the common eigenvalue. Thus $T$ has an eigenvector of the type $(k_i \xi,1)+(y,0)=(k_i \xi+y,1)$ but  $\|k_i\xi+y\|^2=\|k_i\xi\|^2+\|y\|^2>1$ as $\|k_i\xi\|=1$ and $y \neq 0$. So for this case as well, $T$  cannot have a fixed point inside the ball and more than one fixed point on boundary of the ball in view of Remark \ref{R1}. Hence $T$ is parabolic.\\
Let $\lambda_1 \neq \lambda_2$ and $r \neq -1$. This gives that
\begin{eqnarray*}
  \|k_1\xi\|^2 &=&\dfrac{\left| a(r-1)+\sqrt{a^2(r+1)^2-4r}\right| ^2}{4\|\xi\|^2}\\
  &=& \dfrac{\left(a(r-1)+\sqrt{a^2(r+1)^2-4r}\right)\left(a\left(\frac{1}{r}-1\right)+\sqrt{a^2\left(\frac{1}{r}+1\right)^2-\frac{4}{r}}\right)}{4\|\xi\|^2}\\
  &=& \dfrac{\left(a(r-1)+\sqrt{a^2(r+1)^2-4r}\right)\left(a(1-r)+\sqrt{a^2(1+r)^2-4r}\right)}{4r\|\xi\|^2}\\
  &=& \dfrac{-a^2(r-1)^2+a^2(r+1)^2-4r}{4r\|\xi\|^2}\\
  &=& \dfrac{4ra^2-4r}{4r\|\xi\|^2}\\
  &=& 1 \,\, \text{by} \,\,(\ref{eq:7})\\
  &=& \|k_2 \xi\|^2 \,\,\,\,\,\,\text{by}\,\,\text{Lemma}\,\,\ref{n}\,(2).
  \end{eqnarray*}
In view of Remark \ref{R1}, above computation  shows that  $T$ has two fixed points on $\partial{B}$. Also, by the same logic as given for the parabolic case, it cannot have a fixed point inside the ball, hence $T$ is hyperbolic.

  \end{proof}

   Next result explores normal isometries in $G$.
 \begin{theorem}\cite[Theorem 3]{MR}\label{f} Let $T =e^{i \theta}\left[ {\begin{array}{cc}
   UA & U(\xi) \\
   \left<\boldsymbol{\cdot},\xi\right> & a \\
  \end{array} } \right]\in G$, $\theta \in \mathbb{R}$. Then

  \begin{enumerate}
 \item $T$ is normal if and only if $T=e^{i \theta}\left[ {\begin{array}{cc}
   UA & \xi \\
   \left<\boldsymbol{\cdot},\xi\right> & a \\
  \end{array} } \right]$, i.e. $U(\xi)=\xi$.

  \item  $T$ is unitary if and only if  $T=e^{i \theta} \left[ {\begin{array}{cc}
   U & 0 \\
   0 & 1 \\
  \end{array} } \right]$, i.e. $\xi=0$.
  \item If $T$ is normal and $\xi \neq 0$, then  $\sigma(T)=e^{i \theta}\{a \pm \|\xi\| \} \cup \sigma\left(e^{i \theta}U \restriction_{\left<\xi\right>^{\perp}}\right)$, where $a \pm \|\xi\|$ are both positive, different from $1$ and inverses of each other.  Eigenspaces corresponding to the eigenvalues $e^{i \theta}(a \pm \|\xi\|)$ are spanned by the eigenvectors  $\left(\pm \dfrac{\xi}{\|\xi\|},1\right)$ respectively.
  \item In $G$, a unitary element  is elliptic and a non-unitary normal isometry  is hyperbolic.
  \end{enumerate}
  \end{theorem}

 The following lemma is motivated from \cite[Lemma 2.13]{JC}.
 \begin{lemma} \label{b}
  For every $T \in G$, if $M$ is a  subspace of $H \oplus  \mathbb{C}$ such that $T:M \longrightarrow M$ is a bijection,  then  $T$ maps $M^{\dagger}$ to $M^{\dagger}$ bijectively.
  \end{lemma}
  \begin{proof}
 For $u \in M^{\dagger}$ and $v \in M$,
 \begin{align*}
   \mathcal{A}\left(T(u),v\right)&=\mathcal{A}\left(T^{-1}\left(T(u)\right),T^{-1}(v)\right)\\
   &=\mathcal{A}\left(u,T^{-1}(v)\right)\\
   &=0
    \end{align*}
    as the isometry $T : M \longrightarrow M$ is a bijection. Using the similar logic, we can prove that $M^{\dagger}$ is invariant under $T^{-1}$. This proves the result.
       \end{proof}
\begin{proposition}\cite[pp. 20, eq. (21)]{HG}\label{direct sum}
    If $M$ is a non-degenerate, finite dimensional subspace of $H \oplus \mathbb{C}$, then $H\oplus \mathbb{C}=M \oplus M^{\dagger}$.
\end{proposition}
  The next lemma asserts the validity of \cite[Lemma 3.10]{JP} in infinite dimensional set up. It talks about orthogonality of time-like and light-like vectors.
\begin{lemma} \label{parker}
 Let $(x, z),\, (y, w) \in (H \oplus  \mathbb{C}) \setminus \{0\}$ such that $\mathcal{Q}(x, z) \leq 0$ and $\mathcal{Q}(y, w) \leq 0$. Then either $(x, z)=k(y, w)$ for some $k \in \mathbb{C}$ or $\mathcal{A}\left((x, z),(y, w)\right) \neq 0$.
 \end{lemma}
\begin{proof}
  We have $\|x\|^2\,\, \leq |z|^2$ and $\|y\|^2\,\, \leq |w|^2$, $z,\,w \neq 0$. Suppose that \linebreak $ \mathcal{A} \left((x,z),(y,w)\right) = 0$, i.e. $\left<x,y\right>=z \overline{w}$. We will show that $(x,z)$ and $(y,w)$ are linearly dependent.
For arbitrary $k \in \mathbb{C}$,

\begin{eqnarray*}
\left| z-kw\right|^2 &=& (z-kw)(\overline{z}-\overline{kw})\\
&=&  |z|^2-z \overline{k w}-k w\overline{z}+\lvert kw \rvert^2\\
&=& |z|^2-\overline{k}\left<x, y\right>-k\left<y,x\right>+\lvert kw \rvert^2 \\
&\geq& \|x\|^2-\overline{k}\left<x, y\right>-k\left<y,x\right>+\|ky\|^2\\
&=& \left<x-k y, x-k y \right>\\
&=& ||x-k y||^2.
\end{eqnarray*}
Choosing $k=\dfrac{z}{w}$ makes left hand side  of the above inequality  zero and hence right hand side is also $0$ which gives that $x=ky$, i.e.  $(x,z)=(ky,z)=\left((z/w)y,z\right)=(z/w)(y,w)=k(y,w)$.
  \end{proof}
  \begin{theorem}\cite[Corollary 3 - Witt's  Theorem, Section 2, Chapter XV]{HG} \label{wt}
     Let    $(E, \psi)$ be a complex vector space with a Hermitian form $\psi$ and $F$ and $G$ be two finite dimensional subspaces of $E$. If $D:F \longrightarrow G$ is an onto linear isometry, then $D$ can be extended to a surjective linear isometry  of $E$ (see the paragraph below Theorem 2 on page 380 in \cite{HG}).
\end{theorem}

 \section{Conjugacy classes for elliptic and hyperbolic isometries}
 Recall that for a subspace $M$ of $H \oplus \mathbb{C}$, $M^{\dagger}$ denotes the orthogonal complement of $M$ with respect to the Hermitian form $\mathcal{A}$.

\textbf{Proof of Theorem \ref{e}.} Let $h \in Aut(B)$ be the isometry corresponding to $T$ (cf. Theorem \ref{a}) such that  $h$ fixes $b \in B$.\\
 (1)  Using transitivity of elements of $Aut(B)$, there exists an isometry $m \in Aut(B)$ such that $m(0)=b$. Then $m^{-1}h\,m(0)=0$ thereby making $m^{-1}h\,m$  unitary using Lemma \ref{lemma 1} and Lemma \ref{lemma 2}.  Hence $T$ is unitary upto conjugacy.\\
 (2) Let $S$ be the linearization of $m$ and  $S^{-1}TS=\widetilde{U}$ where $\widetilde{U}$ is a unitary operator. As conjugacy preserves spectrum, $\sigma(T)=\sigma(\widetilde{U}) \subseteq S^1$.\\
 (3) Since $b \in B$, $(b,1)$ is a time-like eigenvector of $T$ (cf. Remark \ref{R1}). Conversely, if $(x,z)$ is a time-like eigenvector of an isometry $T' \in Aut(B)$, then $z \neq 0$ as $\|x\| < |z|$. Thus $(z^{-1}x,1)$ is a time-like eigenvector of $T'$ where $\lvert z^{-1}x \rvert <1$ thereby making $T'$ an elliptic isometry using Remark \ref{R1}.\\
 (4) Without loss of generality, a time-like vector is of the form $(x,1)$ where $x \in B$. Using Proposition \ref{direct sum}, $H \oplus \, \mathbb{C}=\left<(x,1)\right> \oplus \, \left<(x,1)\right>^{\dagger}$. By Lemma \ref{b}, $T$ maps ${\left<(x,1)\right>^{\dagger}}$ to ${\left<(x,1)\right>^{\dagger}}$ bijectively. Thus $T=T_1 \oplus T_2$ where $T_1=T\restriction_{\left<(x,1)\right>}$ and $T_2=T\restriction_{\left<(x,1)\right>^{\dagger}}$. Also $\mathcal{Q} \restriction_{\left<(x,1)\right>^{\dagger}}$ is positive definite by Lemma \ref{parker} making $\left<(x,1)\right>^{\dagger}$ an inner product space. We can construct a linear isometry taking $\left<(x,1)\right>$ to $\mathbb{C}$. By Witt's theorem (Theorem \ref{wt}), there exists a linear isometry from $\left<(x,1)\right>^{\dagger}$ onto $\mathbb{C} ^{\dagger}=H$. This makes $\left<(x,1)\right>^{\dagger}$  complete as $H$ is so. Hence $T_2$ is a unitary operator. \qed
\subsection{Proof of Theorem \ref{t1}}
 The following result characterizes non-unitary normal isometries  on the basis of fixed points.
 \begin{lemma} \label{g}
 A non-unitary isometry of $G$ is  normal if and only if it has two  eigenvectors of the type $(\pm x_0,1)$. Moreover  the projections $\pm x_0$ lie on  $\partial{B}$ (cf. Remark \ref{R1}).
 \end{lemma}
 \begin{proof}
Forward part follows from Theorem \ref{f} (3).  For the other way, let $T =\left[ {\begin{array}{cc}
   UA & U(\xi) \\
   \left<\boldsymbol{\cdot}, \xi\right> & a
  \end{array} } \right]$, $\xi \neq 0$, be such that $T(\pm x_0,1)=\lambda_i(\pm x_0,1)$, $i=1,\, 2$ for some $\lambda_1,\,\lambda_2 \in \mathbb{C}$. This gives that $ \pm UA(x_0) +U(\xi)=\pm \lambda_i x_0$ and $\pm \left<x_0,\xi\right>+a=\lambda_i$. Substituting the value of $\lambda_i$ from one equation into the other, we get  $\pm UA(x_0) +U(\xi)=\pm \left(\pm \left<x_0,\xi\right> + a\right)x_0$, i.e.
  \begin{align}
  UA(x_0)+U(\xi)&=\left<x_0,\xi\right>x_0+ax_0\,\,\,\,\text{and} \label{b1}\\
   -UA(x_0)+U(\xi)&=\left<x_0,\xi\right>x_0-ax_0 \label{b2}.
  \end{align}
  Observe that the fact that $T$ is non-unitary forces $x_0$ to be non-zero. On solving (\ref{b1}) and (\ref{b2}) we get $U(\xi)=\left<x_0,\xi\right>x_0$  and $UA(x_0)=ax_0$. Since $\left<x_0,\xi\right>\neq 0$, we further have $U^{-1}(x_0)=\dfrac{\xi}{\left<x_0,\xi\right>}=\dfrac{A(x_0)}{a}$. As $A(\xi)=a\xi$ (cf. Proposition \ref{p}),   $\xi=aA^{-1}(\xi)=\left<x_0,\xi\right>x_0=U(\xi)$ thereby making $T$ a normal isometry (cf. Theorem \ref{f} (1)). Now as $\left<x_0,\xi\right>x_0=\xi$, $\|\xi\|^2=\left<\xi,\xi\right>=\left<\left<x_0,\xi\right>x_0,\xi\right>={\left(\left<x_0,\xi\right>\right)}^2$. Hence $\left<x_0,\xi\right>=\pm \|\xi\|$ and $x_0=\pm \dfrac{\xi}{\|\xi\|} \in \partial{B}$.
\end{proof}

 \begin{lemma} \label{bitransitive}
 $Aut(B)$ acts bi-transitively on $\partial B$.
 \begin{proof}
 The broad approach of the proof is based on  \cite[Proposition 2.1.3]{CG}.
 Let $x_1,\, y_1,\, x_2, y_2 \in \partial B$ be such that $x_1 \neq x_2$ and $y_1 \neq y_2$. We have to show the existence of an element $m \in Aut(B)$ satisfying $m(x_1)=y_1$ and $m(x_2)=y_2$. Clearly the elements $(x_i,1),\, (y_i,1),\,\,i=1,2$ are light-like vectors. Let $W=span\,\{(x_1,1),(x_2,1)\}$. Define a linear map $T:W \longrightarrow H \oplus \, \mathbb{C}$ as
 \[T(x_1,1)=\mu(y_1,1)\]
 \[T(x_2,1)=(y_2,1)\]
 where $\mu=\dfrac{\mathcal{A}\Big((x_1,1),(x_2,1) \Big)}{\mathcal{A}\Big((y_1,1),(y_2,1) \Big)}$ which makes sense by Lemma \ref{parker}. We will see that  with this choice of $\mu$, $T$ becomes an isometry. As $H \oplus \mathbb{C}$ is a complex Hilbert space, it is enough to show that $\mathcal{Q}(T(\mathbb{x})) =\mathcal{Q}(\mathbb{x})$ for every $\mathbb{x} \in W$.  Further as $(x_1,1)$ and $(x_2,1)$ are light-like vectors,  it suffices to  show that
 $\mathcal{A}\Big(T(x_1,1),T(x_2,1)\Big)=\mathcal{A}((x_1,1),(x_2,1))$. Consider
 \begin{align*}
 \mathcal{A}\Big(T(x_1,1),T(x_2,1)\Big)
& =\mathcal{A}(\mu(y_1,1),(y_2,1))\\
 &=\mu\mathcal{A}((y_1,1),(y_2,1))\\
&=\mathcal{A}((x_1,1),(x_2,1)).
 \end{align*}
 Using Witt's theorem  (Theorem \ref{wt}), $T$ can be extended to an isometry $\widetilde{T}$ of $H \oplus \, \mathbb{C}$.  So, we have an element $\widetilde{T} \in G$. Thus it is of the form some $e^{i \theta} \left[ {\begin{array}{cc}
   UA & U(\xi) \\
   \left<\boldsymbol{\cdot},\xi\right> & a \\
  \end{array} } \right]$ such that $e^{i \theta} \left[ {\begin{array}{cc}
   UA & U(\xi) \\
   \left<\boldsymbol{\cdot},\xi\right> & a \\
  \end{array} } \right] \left[ {\begin{array}{c}
   x_i \\
   1 \\
  \end{array} } \right]=\mu_i\left[ {\begin{array}{c}
   y_i \\
   1 \\
  \end{array} } \right]$ where $\mu_1=\mu$ and $\mu_2=1$. This  gives that  $e^{i \theta}\Big(UA(x_i)+U(\xi)\Big)=\mu_i y_i$ and $e^{i \theta}\Big(\left<x_i,\xi\right>+a\Big)=\mu_i$ which implies that $UA(x_i)+U(\xi)=\Big(\left<x_i,\xi\right>+a\Big)y_i$.  We have $x \mapsto \dfrac{UA(x)+U(\xi)}{\left<x,\xi\right>+a}=m(x)$ say,  is the corresponding isometry in $Aut(B)$ satisfying $m(x_i)=y_i,\,\,i=1,2.$
  \end{proof}
 \end{lemma}
  \textbf{Proof of Theorem \ref{t1}.}  Let $h \in Aut(B)$ be the isometry corresponding to $T$ which fixes $y_1,\,y_2 \in \partial{B}$.\\
 (1)   Lemma \ref{bitransitive} says that  there exists $m \in Aut(B)$ such that for $\pm x \in \partial {B}$, $m(y_1)=x$ and $m(y_2)=-x$. This implies the isometry $m\,h\,m^{-1}$ fixes $\pm x$. It is non-unitary being conjugate to a hyperbolic isometry and hence its linearization is a   normal isometry by Remark \ref{R1} and Lemma \ref{g}. \\
 (2) Let $R$ be a linearization of $m\,h\,m^{-1}$. So $R$ is a non-unitary normal isometry. As conjugacy preserves spectrum, $\sigma(T)=\sigma(R)$. Rest follows  by part (3) of Theorem \ref{f}.\\
 (3) By Remark \ref{R1}, $(y_1,1)$ and $(y_2,1)$ are eigenvectors of $T$.\\ Let $M=span\,\{(y_1,1),(y_2,1)\}$. Then $M$ is non-degenerate. This gives that  $H \oplus \, \mathbb{C}=M \oplus \, M^{\dagger}$ using Proposition \ref{direct sum}. Also Lemma \ref{b} implies that $T$ maps  $M^{\dagger}$ to  $M^{\dagger}$ bijectively, hence   $T=T_1 \oplus \, T_2$ where  $T_1=T \restriction_{M}$ and $T_2=T\restriction_{M^{\dagger}}$. By Lemma \ref{parker}, $\mathcal{Q} \restriction_{M^{\dagger}}$ is positive definite making $M^{\dagger}$ an inner product space. For $x \in \partial{B}$, we can construct a linear isometry mapping $\left<(y_1,1)\right>$ to $\left<(x,1)\right>$ and $\left<(y_2,1)\right>$ to $\left<(-x,1)\right>$. Using Witt's theorem (Theorem \ref{wt}), there exists a linear isometry from $M^{\dagger}$ onto $(\left<x\right> \oplus \mathbb{C})^{\dagger}$. As $(\left<x\right> \oplus \mathbb{C})^{\dagger}=(\left<x\right> \oplus \mathbb{C})^{\perp}$ and $(\left<x\right> \oplus \mathbb{C})^{\perp}$ is complete, this makes $M^{\dagger}$ complete and hence   $T_2$ is a unitary operator.\qed
\section{Conjugacy classes for parabolic isometries}
 \subsection{Stabilizer of infinity}
 We will be following \cite{CG} in our approach and study parabolic isometries on Siegel domain model. \\
 Let $e \in \partial{B}$. We write $H \oplus \, \mathbb{C}=\left<e\right> \oplus \, \left<e\right>^{\perp} \oplus \, \mathbb{C}$ where $\left<e\right>^{\perp}$ has $\{e_j,\,\,j \in I\} $ as  orthonormal basis where $I$ is an indexing set. A  general element $(x,z) \in H \oplus \, \mathbb{C}$ has the representation $(x,z)=\left<x,e\right>(e,0)+(x',0)+z(0,1)$. \\
 In the following, we shall describe the Siegel domain model along with the isometries of it and then examine those isometries of the model which fix $\infty$.\\
 Consider the sesquilinear form
\begin{eqnarray}\label{eq:10}
\widehat{\mathcal{A}}\left((x,z),(y,w)\right)=-z\overline{\left<y,e\right>}-\left<x,e\right>\overline{w}+\left<x',y'\right>.
\end{eqnarray}
Observe that $\widehat{\mathcal{A}}\left((x,z),(y,w)\right)=\left<\widehat{A'}(x,z),(y,w)\right>$ where
\begin{eqnarray*}
\widehat{A'}(x,z)=(-ze+x',-\left<x,e\right>).
\end{eqnarray*}
Hence
\begin{eqnarray}\label{r}
\widehat{A'} =\left[\begin{array}{ccc}
0 & 0 & -1 \\
0 & I & 0\\
-1 & 0 & 0
\end{array}\right]
\end{eqnarray}
Let $\widehat{\mathcal{Q}}$ be the quadratic form determined by $\widehat{\mathcal{A}}$. So
\begin{eqnarray}\label{eq:11}
\widehat{\mathcal{Q}}(x,z)=-2\text{Re}\,(\overline{z}\left<x,e\right>)+\|x'\|^2.
\end{eqnarray}
If $\widehat{\mathcal{Q}}(x,z)<0$, i.e. $\text{Re}\,(\overline{z}\left<x,e\right>)>\dfrac{1}{2}\|x'\|^2$ then $z \neq 0$. Hence the projective space induced from the negative vectors is given by 
\[ \Sigma=\left\{x \in H \,\,:\,\,\text{Re}\,\left<x,e\right> >\dfrac{1}{2}\|x'\|^2\right\}.\]
 We call it the Siegel domain.
Consider a linear transformation $D$ on $H \oplus \mathbb{C}$ which takes
 $(e,0)$ to $\dfrac{(e,1)}{\sqrt{2}}$, $(0,1)$ to $\dfrac{(-e,1)}{\sqrt{2}}$ and is identity on the rest. Notice that such $D$ is unitary as it takes an orthonormal basis to an orthonormal basis. The operator $D$ has the following representation.
 \begin{eqnarray*}
 D=\left[\begin{array}{ccc}
{}\dfrac{1}{\sqrt{2}} & {}0 & {}-\dfrac{1}{\sqrt{2}}\\
{}0 & {}I & {}0 \\
 {}\dfrac{1}{\sqrt{2}} & {}0 & {}\dfrac{1}{\sqrt{2}}
 \end{array}\right].
 \end{eqnarray*}
  As $D$ is unitary, $D^{-1}=D^*=\left[\begin{array}{ccc}
{}\dfrac{1}{\sqrt{2}} & {}0 & {}\dfrac{1}{\sqrt{2}}\\
{}0 & {}I & {}0 \\
 {}-\dfrac{1}{\sqrt{2}} & {}0 & {}\dfrac{1}{\sqrt{2}}
 \end{array}\right]$.
\begin{proposition}\label{equivalence}
The Hermitian forms $\mathcal{A}$ and $\widehat{\mathcal{A}}$ are equivalent, i.e. $D^{-1}A'D=\widehat{A'}$ (see (\ref{A'}) and (\ref{r})).
\end{proposition}
\begin{proof}
\begin{eqnarray*}
D^{-1}A'D&=&\left[\begin{array}{ccc}
{}\dfrac{1}{\sqrt{2}} & {}0 & {}\dfrac{1}{\sqrt{2}}\\
{}0 & {}I & {}0 \\
 {}-\dfrac{1}{\sqrt{2}} & {}0 & {}\dfrac{1}{\sqrt{2}}
 \end{array}\right]\left[\begin{array}{ccc}
1 & 0 & 0 \\
0 & I & 0\\
0 & 0 & -1
\end{array}\right]\left[\begin{array}{ccc}
{}\dfrac{1}{\sqrt{2}} & {}0 & {}-\dfrac{1}{\sqrt{2}}\\
{}0 & {}I & {}0 \\
 {}\dfrac{1}{\sqrt{2}} & {}0 & {}\dfrac{1}{\sqrt{2}}
 \end{array}\right]\\
 &=&\left[\begin{array}{ccc}
{}\dfrac{1}{\sqrt{2}} & {}0 & {}-\dfrac{1}{\sqrt{2}}\\
{}0 & {}I & {}0 \\
 -{}\dfrac{1}{\sqrt{2}} & {}0 & -{}\dfrac{1}{\sqrt{2}}
 \end{array}\right]\left[\begin{array}{ccc}
{}\dfrac{1}{\sqrt{2}} & {}0 & {}-\dfrac{1}{\sqrt{2}}\\
{}0 & {}I & {}0 \\
 {}\dfrac{1}{\sqrt{2}} & {}0 & {}\dfrac{1}{\sqrt{2}}
 \end{array}\right]\\
&=& \left[\begin{array}{ccc}
0 & 0 & -1 \\
0 & I & 0\\
-1 & 0 & 0
\end{array}\right].
\end{eqnarray*}
\end{proof}

Next we will see that the projectivization of $D^{-1}$ is a Cayley map which is a biholomorphism between the unit ball $B$  and the Siegel domain $\Sigma$.\\
For a time-like vector $(x,1)$,
\begin{eqnarray*}
 D^{-1}(x,1)&=& \left(\dfrac{\left<x,e\right>e}{\sqrt{2}}+x'+\dfrac{e}{\sqrt{2}},-\dfrac{\left<x,e\right>}{\sqrt{2}}+\dfrac{1}{\sqrt{2}}\right)\\
 &=& \dfrac{\left<x,e\right>+1}{\sqrt{2}}(e,0)+(x',0)+\dfrac{1-\left<x,e\right>}{\sqrt{2}}(0,1)\\
 &=& \dfrac{1-\left<x,e\right>}{\sqrt{2}}\left(\dfrac{1+\left<x,e\right>}{1-\left<x,e\right>}(e,0)+\dfrac{\sqrt{2}}{1-\left<x,e\right>}(x',0)+(0,1)\right)
 \end{eqnarray*}
which is well defined as Cauchy-Schwarz inequality gives $|\left<x,e\right>| \leq \|x\| \|e\|=\|x\|<1$, hence $1-\left<x,e\right> \neq 0$.  $D^{-1}$ induces the following projective transformation
 \[\mathcal{D}(x)=\dfrac{1+\left<x,e\right>}{1-\left<x,e\right>}e+\dfrac{\sqrt{2}}{1-\left<x,e\right>}x'.\]
 \begin{proposition}
 The projective map  $\mathcal{D}: B \longrightarrow \Sigma$ defined by $x \mapsto \dfrac{1+\left<x,e\right>}{1-\left<x,e\right>}e+\dfrac{\sqrt{2}}{1-\left<x,e\right>}x'$ is a bi-holomorphism.
 \end{proposition}
 \begin{proof}
 First we will show that $\mathcal{D}(B) \subseteq \Sigma$. Consider for $x \in B$
 \begin{eqnarray*}
 \text{Re}\,\left(\dfrac{1+\left<x,e\right>}{1-\left<x,e\right>}\right)&=&\text{Re}\,\left(\left(\dfrac{1+\left<x,e\right>}{1-\left<x,e\right>}\right)\left(\dfrac{1-\left<e,x\right>}{1-\left<e,x\right>}\right)\right)\\
 &=&\text{Re}\,\left(\dfrac{1-|\left<x,e\right>|^2+2i\,\text{Im}\,\left<x,e\right>}{\lvert 1-\left<x,e\right>
\rvert^2}\right)\\
&=& \dfrac{1-|\left<x,e\right>|^2}{\lvert 1-\left<x,e\right>
\rvert^2}.
\end{eqnarray*}
As $x \in B$, $\|x\|^2=\|\left<x,e\right>e+x'\|^2=|\left<x,e\right>|^2+\|x'\|^2 < 1$, i.e. $1-|\left<x,e\right>|^2 > \|x'\|^2$. Hence
\[\text{Re}\,\left(\dfrac{1+\left<x,e\right>}{1-\left<x,e\right>}\right)>\dfrac{\|x'\|^2}{\lvert1-\left<x,e\right>\rvert^2}=\dfrac{1}{2}\left\|\dfrac{\sqrt{2}}{1-\left<x,e\right>}x'\right\|^2.\]
As $D^{-1}(x,1)=\alpha(x)(\mathcal{D}(x),1)$ for some $\alpha(x) \in \mathbb{C}$, bijectivity of $D^{-1}$ implies that of $\mathcal{D}$. The holomorphicity of the map is easy to establish.

\end{proof}
\begin{remark}\label{infinity}
Let $\mathcal{D}(e)$ be denoted by $\infty$ and  $\overline{\Sigma}=\Big\{x \in H \,\,:\,\, \text{Re}\,\left<x,e\right> \geq \frac{1}{2}\|x'\|^2\Big\}\cup\{\infty\}$. So, $\overline{\Sigma}$ is the completion of $\overline{\Sigma} \setminus \infty$.  Observe that if for $x \in \partial{B}$, $\left<x,e\right>=1$, then  $1=|\left<x,e\right>| \leq \|x\|=1$ and hence $x=ce$, for some $c \in \mathbb{C}$, by Cauchy-Schwarz inequality. This gives that $x=e$ and we have $1-\left<x,e\right> \neq 0$ for $x \in \partial{B} \setminus \{e\}$.
 Proof of the above Proposition tells that $\mathcal{D}$ maps $\partial{B}$ to $\partial{\Sigma}$ bijectively. Also $\mathcal{D}:\overline{B}\setminus \{e\} \longrightarrow \overline{\Sigma}\setminus \{\infty\}$ is a homeomorphism. We define neighbourhoods of $\infty$ to be images of neighbourhoods of $e$. This makes $\overline{\Sigma}$ a topological space and $\mathcal{D}:\overline{B} \longrightarrow \overline{\Sigma}$, a homeomorphism.
The projective transformation $\mathcal{D}$ is called the Cayley map. Metric on $\Sigma$ can be defined as the pull back of the Carath\'eodory metric on $B$ via Cayley map.
\end{remark}

Let $\widehat{G}$ be the group of all bijective bounded linear operators on $H \oplus \mathbb{C}$ leaving $\widehat{\mathcal{A}}$ invariant. Proposition \ref{equivalence} tells that
\[\widehat{G}=\{\widehat{T} \in B(H \oplus \mathbb{C})\,\,:\,\,\widehat{T}\,\,\text{is bijective and}\,\,\widehat{T}=D^{-1}TD,\,\,T \in G\}.\]
Also elements of the group $Aut(\Sigma)$ are the projective maps induced from elements of $\widehat{G}$ where the group $Aut(\Sigma)$ is defined as follows.
 \[Aut(\Sigma) =\{\mathcal{D}h\mathcal{D}^{-1},\,\,\,\,h \in Aut(B)\}.\]
\begin{remark}\label{relation}
Cayley map gives fixed point classification for isometries in $Aut(\Sigma)$.
      We call a vector $(x,z)$  time-like, light-like or space-like  with respect to the Hermitian form $\w{\mathcal{A}}$ defined in (\ref{eq:10}) if
     $\w{Q}(x,z) <,= \text{or} > 0$ respectively. We shall refer to these vectors  also as time-like, light-like or space-like. Meaning would be clear from the context.

 \end{remark}

\textbf{Let us denote  the stabilizer group of infinity in $\boldsymbol{\widehat{G}}$ by $\boldsymbol{\widehat{G}_{\infty}}$}.  $\mathcal{D}(e)=\infty$. $(e,1)$ is the lift of $e$ and $D^{-1}(e,1)={\sqrt{2}}(e,0)$. Thus $(e,0)$ will be considered as standard lift of $\infty$.
 \begin{proposition}[Stabilizer group of infinity] \label{l}
 Let $\widehat{T} \in \widehat{G}$ satisfy $\widehat{T}(e,0)=\lambda(e,0)$, then
 \begin{eqnarray}\label{eq:1}
 \widehat{T}=\left[\begin{array}{ccc}
  \lambda    & \left<\boldsymbol{\cdot},U^{-1}(a')\right> &   \mu\,s\\
    0  & U & \mu\,a'\\
    0 & 0 & \mu
 \end{array}\right]
 \end{eqnarray}
 where  $a' \in \left<e\right>^{\perp}$ in $H$, $U \in \mathcal{U}(\left<e\right>^{\perp})$, $\lambda,\, s \in \mathbb{C}$  obey $\lambda \overline{\mu}=1$ and $Re\, s=\frac{1}{2}\|a'\|^2$.
 \end{proposition}
 \begin{proof}
 The fact that $\w{T}$ leaves $\w{\mathcal{A}}$ invariant will be used to establish the result. Let $\widehat{T}(e_j,0) =\left(b_j\,e+{w_j}' ,z_j\right)$ and $\widehat{T}(0,1)=\left(te+c',\mu\right)$,   where $b_j, \, z_j,\,t,\,\mu \in \mathbb{C}$ and ${w_j}',\,c' \in \left<e\right>^{\perp},\,j \in I$. 
Using (\ref{eq:10}) and (\ref{eq:11}), we get $\widehat{{Q}}(e,0)=\widehat{{Q}}(0,1)=0$, $\widehat{{Q}}(e_j,0)=1$ and $\widehat{\mathcal{A}}\Big((e,0),(e_j,0)\Big)=\widehat{\mathcal{A}}\Big((0,1),(e_j,0)\Big)=\widehat{\mathcal{A}}\Big((e_i,0),(e_j,0)\Big)\\=0$, $i \neq j$ and $\widehat{\mathcal{A}}\Big((e,0),(0,1)\Big)=-1$.

Also $\widehat{{Q}}\Big(\widehat{T}(e,0)\Big)=\widehat{{Q}}\Big(\lambda(e,0)\Big)=0$. Since $\widehat{T}$ leaves $\widehat{\mathcal{A}}$ invariant, we have\\ $\widehat{{Q}}\Big(\widehat{T}(0,1)\Big)=\widehat{{Q}}(0,1)$, i.e. $\widehat{{Q}}\left(te+c',\mu\right)=-2\text{Re}\,(\overline{\mu}t)+\|c'\|^2=0$ which gives that
\[\text{Re}\,(\overline{\mu}t)=\dfrac{1}{2}\|c'\|^2.\]
Similarly $\widehat{\mathcal{A}}\Big(\widehat{T}(e,0),\widehat{T}(0,1)\Big)=\widehat{\mathcal{A}}\Big((e,0),(0,1)\Big)$, i.e.\\ $\widehat{\mathcal{A}}\Big(\lambda(e,0),\left(te+c',\mu\right)\Big) =-\lambda \overline{\mu}=-1$ implying
\[\lambda \overline{\mu}=1.\]
 $\widehat{\mathcal{A}}\Big(\widehat{T}(e,0),\widehat{T}(e_j,0)\Big)=\widehat{\mathcal{A}}\Big((e,0),(e_j,0)\Big)$, i.e. $\Big(\lambda(e,0),\left(b_j\,e+{w_j}'  ,z_j\right)\Big)=-\lambda \overline{z_j}=0$ giving that
 \[z_j=0.\]
 $\widehat{\mathcal{A}}\Big(\widehat{T}(0,1),\widehat{T}(e_j,0)\Big)=\widehat{\mathcal{A}}\Big((0,1),(e_j,0)\Big)$ , i.e.\\ $\widehat{\mathcal{A}}\Big(\left(te+c',\mu\right), \left(b_j\,e+{w_j}',0\right)\Big)=-\mu \overline{b_j}+  \left<c',{w_j}'\right>=0$ which yields that $\overline{\mu}b_j=\left<{w_j}',c'\right>$, i.e.
 \begin{eqnarray}\label{bj}
b_j= \lambda \left<{w_j}',c'\right>.
\end{eqnarray}
$\widehat{{Q}}\Big(\widehat{T}(e_j,0)\Big)=\widehat{{Q}}(e_j,0)$, i.e. $\widehat{{Q}}\left(b_j\,e+{w_j}', 0\right)=\|{w_j}'\|^2=1$.

Lastly for $i \neq j$,  $\widehat{\mathcal{A}}\Big(\widehat{T}(e_i,0),\widehat{T}(e_j,0)\Big)=\widehat{\mathcal{A}}\Big((e_i,0),(e_j,0)\Big)$, i.e.\\ $\widehat{\mathcal{A}}\Big(\left(b_i\,e+{w_i}' ,0\right),\left(b_j\,e+{w_j}' ,0\right)\Big)=\left<{w_i}',{w_j}'\right>=0$.

The above computation asserts  that the collection $\{{w_j}',\,\,j \in I\}$ is   orthonormal in $\left<e\right>^{\perp}$.\\
Define a linear operator $U : \left<e\right>^{\perp} \longrightarrow \left<e\right>^{\perp}$ as $U(e_j)= {w_j}'$. Then $U$ is bounded and norm preserving.
Let $x \in H$ where $x'=\sum\limits_{j=1_x}^{n_x}\left<x',e_j\right>e_j$. Observe that
\begin{align*}
\sum\limits_{j=1_x}^{n_x}\left<x',e_j\right>b_j&= \sum\limits_{j=1_x}^{n_x}\left<x',e_j\right>\lambda\left<U(e_j),c'\right>\,\,by\,\, (\ref{bj})\\
&=\lambda\left<U\left(\sum\limits_{j=1_x}^{n_x}\left<x',e_j\right>e_j\right),c'\right>\\
&=\lambda\left<U(x'),c'\right>\\
&=\lambda\left<x',U^{-1}(c')\right>.
\end{align*}
Consider
\begin{align*}
\widehat{T}(x,z)&=\widehat{T}\Big(\left<x,e\right>(e,0)+(x',0)+z(0,1)\Big)\\
&= \left(\lambda\,\left<x,e\right>  (e,0)+\sum\limits_{j=1_x}^{n_x}\left<x',e_j\right>\Big(b_j\,e+U(e_j) ,0\Big)+z\Big(te+c',\mu\Big)\right)\\
&=  \left(\Big(\lambda\left<x,e\right>+\sum\limits_{j=1_x}^{n_x}\left<x',e_j\right>b_j+z\,t\Big)e + \sum\limits_{j=1_x}^{n_x}\left<x',e_j\right>U(e_j)+zc' ,z \mu\right) \\\
&=  \bigg(\Big(\lambda\left<x,e\right>+\lambda\left<x',U^{-1}(c')\right>+z\,t\Big)e + U(x')+zc' ,z \mu\bigg)\\
&= \left[\begin{array}{ccc}
  \lambda    & \lambda\left<\boldsymbol{\cdot}, U^{-1}(c')\right> &   t\\
    0  & U & c'\\
    0 & 0 & \mu
 \end{array}\right]\left[\begin{array}{c}
 \left<x,e\right>\\
 x'\\
 z
 \end{array}\right].
\end{align*}
We will see that $U$ is surjective which will make it a unitary operator. As $\widehat{T}$ is surjective, for $y' \in \left<e\right>^{\perp}$, there exists $(x,z) \in H \oplus  \mathbb{C}$ such that $\widehat{T}(x,z)=y'$. This implies that $U(x')+zc'=y'$ and $z \mu=0$. This further yields that $z=0$ and hence $U(x')=y'$ thereby making $U$ a unitary operator.\\
 $\widehat{T}$ can also be written as
 $\left[\begin{array}{ccc}
  \lambda    & \lambda\left<\boldsymbol{\cdot}, U^{-1}\left(\dfrac{\mu c'}{\mu}\right)\right> &   \dfrac{\mu t}{\mu}\\
   0  & U & \dfrac{\mu c'}{\mu}\\
    0 & 0 & \mu
 \end{array}\right]$.
  By taking $\dfrac{c'}{\mu}=a'$ and $\dfrac{t}{\mu}=s$, we get
 \[\widehat{T}=\left[\begin{array}{ccc}
  \lambda    & \left<\boldsymbol{\cdot}, U^{-1}(a')\right> &   \mu\,s\\
    0  & U & \mu\,a'\\
    0 & 0 & \mu
 \end{array}\right]\]
 where ${\lambda}\overline{\mu}=1$ and $\text{Re}(\overline{\mu}\mu s)=\dfrac{1}{2}\|\mu a'\|^2$, i.e. $\text{Re}\,s=\dfrac{1}{2}\|a'\|^2$.

 \end{proof}

\begin{corollary}\label{c6}
If $\widehat{T} \in \widehat{G}_\infty$ is   such that $s=0$, cf. (\ref{eq:1}) then $\widehat{T}$ is non-parabolic.
\end{corollary}
\begin{proof}
Observe from Proposition \ref{l} that if $s=0$,  $a'=0$, so $\widehat{T}=\left[\begin{array}{ccc}
  \lambda    & 0 &   0\\
    0  & U & 0\\
    0 & 0 & \mu
 \end{array}\right]$ which  has $(0,1)$ and $(e,0)$ as light-like eigenvectors.  Hence $\w{T}$ is non-parabolic.
\end{proof}
\begin{remark}\label{R2}
  In  Proposition \ref{l}, $\lambda \overline{\mu}=1$, i.e ${\mu}=\dfrac{1}{\overline{\lambda}}$. This gives that  $|\lambda|=1$ if and only if $\lambda=\mu$.
   \end{remark}
 \begin{corollary}[Iwasawa decomposition]
 Every element in $\widehat{G}_{\infty}$ is the product of  a translation $(1,a',s)=\left[\begin{array}{ccc}
  1    &  \left<\boldsymbol{\cdot},a'\right> &   s\\
    0  & I & a'\\
    0 & 0 & 1\\
 \end{array}\right]$ ($\text{Re}\,s=\dfrac{1}{2}\|a'\|^2$), a rotation $ R_U=\left[\begin{array}{ccc}
  1    &  0 &   0\\
    0  & U & 0\\
    0 & 0 & 1
 \end{array}\right]$ where $U$ is a unitary operator and a dilation $D_{\lambda}=\left[\begin{array}{ccc}
  \lambda    &  0 &   0\\
    0  & I & 0\\
    0 & 0 & \mu
 \end{array}\right]$ where $\lambda \overline{\mu}=1$.
 \end{corollary}

 \begin{proof}
 We have $(1,a',s)R_U D_{\lambda}=\left[\begin{array}{ccc}
  1    & \left<\boldsymbol{\cdot},a'\right> &   s\\
    0  &  I & a'\\
    0 & 0 & 1
 \end{array}\right]\left[\begin{array}{ccc}
  1    & 0 &   0\\
    0  & U & 0\\
    0 & 0 & 1
 \end{array}\right]\left[\begin{array}{ccc}
  \lambda    & 0 &   0\\
    0  &  I & 0\\
    0 & 0 & \mu
 \end{array}\right]$\\
 $=\left[\begin{array}{ccc}
  1    & \left<\boldsymbol{\cdot},U^{-1}(a')\right> &   s\\
    0  & U & a'\\
    0 & 0 & 1
 \end{array}\right]\left[\begin{array}{ccc}
  \lambda    & 0 &   0\\
    0  &  I & 0\\
    0 & 0 & \mu
 \end{array}\right]=\left[\begin{array}{ccc}
  \lambda    & \left<\boldsymbol{\cdot},U^{-1}(a')\right> &   \mu\,s\\
    0  & U & \mu\,a'\\
    0 & 0 & \mu
 \end{array}\right]$.
 \end{proof}
 \begin{proposition}\label{p10}
 The projective transformation induced by $\widehat{T} \in \widehat{G}_\infty$ satisfying (\ref{eq:1}) is an affine map $\mathcal{\widehat{T}}:\Sigma \longrightarrow \Sigma$ given by
 \[\mathcal{\widehat{T}}(x) = \overline{\lambda}\left[\begin{array}{cc}
  \lambda    & \left<\boldsymbol{\cdot},U^{-1}(a')\right>\\
    0  & U \\
 \end{array}\right]\left[\begin{array}{c}
  \left<x,e\right>\\
  x'\\
 \end{array}\right]+\left[\begin{array}{c}
    s\\
 a'\\
 \end{array}\right].\]
 \end{proposition}
 \begin{proof}
 Since the $\mathbb{C}$-component of a time-like vector with respect to $\w{Q}$ is non-vanishing, so without loss of generality, let $(x,1)$ be a time-like vector  with respect  to $\w{Q}$, i.e.   $x \in \Sigma$. Then $\widehat{T}(x,1)=\bigg(\Big(\lambda\left<x,e\right>+\left<U(x'),a'\right>+\mu s\Big)e+U(x')+\mu a',\mu\bigg)=\mu\Bigg(\bigg(|\lambda|^2\left<x,e\right>+\overline{\lambda}\left<U(x'),a'\right>+s\bigg)e+\overline{\lambda}U(x')+a',1\Bigg)$.  
 Hence the projective map say $\mathcal{\widehat{T}}:\Sigma \longrightarrow \Sigma$, induced by $\widehat{T}$ is given by
 \begin{align*}
 \mathcal{\widehat{T}}(x)&=\bigg(|\lambda|^2\left<x,e\right>+\overline{\lambda}\left<U(x'),a'\right>+s\bigg)e+\overline{\lambda}U(x')+a'\\
 &=\overline{\lambda}\left[\begin{array}{cc}
  \lambda    & \left<\boldsymbol{\cdot},U^{-1}(a')\right>\\
    0  & U \\
 \end{array}\right]\left[\begin{array}{c}
  \left<x,e\right>\\
  x'\\
 \end{array}\right]+\left[\begin{array}{c}
    s\\
 a'\\
 \end{array}\right].
 \end{align*}
 \end{proof}

 Now we compute the spectrum of the isometry obtained in Proposition \ref{l}.
\begin{lemma}\label{l3}
Let $\widehat{T} \in \widehat{G}_\infty$ be of the form (\ref{eq:1}) such that $s \neq 0$.  Then $\sigma(\widehat{T})\setminus \{\lambda,\mu\}=\sigma(U)\setminus \{\lambda,\mu\}.$
\end{lemma}
\begin{proof}
We will show that for any  $r \in \mathbb{C} \setminus \{\lambda, \mu \}$, $\widehat{T}-rI$ is bijective if and only if $U-rI$ is bijective.
Let $r \in \mathbb{C} \setminus \{\lambda, \mu \}$. If $U-rI$ is one-one, $\widehat{T}-rI$ is one-one, for \\
 $(\widehat{T}-r I)(x,z)=\left[\begin{array}{ccc}
  \lambda-r    & \left<\boldsymbol{\cdot},U^{-1}(a')\right> &   \mu s\\
    0  & U-r I & \mu a'\\
    0 & 0 & \mu-r
 \end{array}\right]\left[\begin{array}{c}
 \left<x,e\right>\\
 x'\\
 z
 \end{array}\right]
 \equiv 0$ implies $z=0$, $x'=0$ as $U-rI$ is injective and hence $\left<x,e\right>=0$. Conversely, suppose $\widehat{T}-rI$ is one-one and $(U-rI)(x')=0$ for some $x' \in \left<e\right>^{\perp}$. We have \\
 $(\widehat{T}-r I)\left({\dfrac{\left<x',U^{-1}(a')\right>}{r-\lambda}}e+x',0\right)=\begin{bmatrix}
  \lambda-r    & \left<\boldsymbol{\cdot},U^{-1}(a')\right> &   \mu s\\
    0  & U-r I & \mu a'\\
    0 & 0 & \mu-r
 \end{bmatrix}\begin{bmatrix}
\dfrac{\left<x',U^{-1}(a')\right>}{r-\lambda} \\
 x'\\
 0
 \end{bmatrix}
 \equiv 0$ which implies $x'=0$. Thus $U-rI$ is one-one.
Now, suppose $U-rI$ is onto. We will see that $\widehat{T}-rI$ is onto. Let $\Big(\left<y,e\right>e+y',w\Big) \in H \oplus  \mathbb{C}$. Since $U-rI$ is onto, for $y'-\dfrac{\mu wa'}{\mu-r} \in \left<e\right>^{\perp}$, there exists $x' \in \left<e\right>^{\perp}$ such that $(U-rI)(x')=y'-\dfrac{\mu wa'}{\mu-r}$. Let $\left<x,e\right>=\dfrac{\left<y,e\right>-\dfrac{\mu sw}{\mu-r}-\left<x',U^{-1}(a')\right>}{\lambda-r}$. With this choice, consider\\
 $(\widehat{T}-r I)\left(x,\dfrac{w}{\mu-r}\right)=\left[\begin{array}{ccc}
  \lambda-r    & \left<\boldsymbol{\cdot},U^{-1}(a')\right> &   \mu s\\
    0  & U-r I & \mu a'\\
    0 & 0 & \mu-r
 \end{array}\right]\left[\begin{array}{c}
\left<x,e\right> \\
 x'\\
\dfrac{w}{\mu-r} \\
 \end{array}\right]=\\
 \Bigg(\bigg((\lambda-r)\left<x,e\right>+\left<x',U^{-1}(a')\right>+\dfrac{\mu sw}{\mu-r}\bigg)e+(U-rI)(x')+\dfrac{\mu wa'}{\mu-r},w\Bigg)=\Big(\left<y,e\right>e+y',w\Big)$. Conversely, let  $\widehat{T}-rI$ be surjective and $y' \in \left<e\right>^{\perp}$. Then $(y',0) \in H \oplus  \mathbb{C}$ and by surjectivity of $\widehat{T}-rI$ there exists $(x,z)=\Big(\left<x,e\right>e+x',z\Big) \in H \oplus \mathbb{C}$ such that
 \begin{align*}
 (\widehat{T}-r I)(x,z)&=\left[\begin{array}{ccc}
  \lambda-r    & \left<\boldsymbol{\cdot},U^{-1}(a')\right> &  \mu s\\
    0  & U-r I & \mu a'\\
    0 & 0 & \mu-r
 \end{array}\right]\left[\begin{array}{c}
 \left<x,e\right>\\
 x'\\
 z
 \end{array}\right]\\
 &=\left[\begin{array}{c}
 (\lambda-r)\left<x,e\right>+\left<x',U^{-1}(a')\right>+\mu sz\\
 (U-rI)x'+\mu a'z\\
 (\mu-r)z
 \end{array}\right]=\left[\begin{array}{c}
 0\\
 y'\\
 0
 \end{array}\right].
 \end{align*}
  This implies $z=0$ ($\mu \neq r$) and hence $(U-rI)(x')=y'$ making $U-rI$  surjective. This  proves the lemma.

\end{proof}
 \begin{proposition} \label{m}
If $\widehat{T} \in \widehat{G}_\infty$ is of the form (\ref{eq:1}) such that $s \neq 0$, then $\sigma(\widehat{T})=\{\lambda, \mu\} \cup \sigma(U)$.
 \end{proposition}
 \begin{proof}
Clearly, $\lambda \in \sigma(\widehat{T})$. For $\mu \neq \lambda$, we will see that $\mu$ is also an eigenvalue of $\widehat{T}$, i.e.  ${\text{ker}}(\widehat{T}-\mu I) \neq \{0\}$.
 \[(\widehat{T}-\mu I)(x,z)=\left[\begin{array}{ccc}
  \lambda-\mu    & \left<\boldsymbol{\cdot},U^{-1}(a')\right> &   \mu s\\
    0  & U-\mu I & \mu a'\\
    0 & 0 & 0\\
 \end{array}\right]\left[\begin{array}{c}
 \left<x,e\right>\\
 x'\\
 z\\
 \end{array}\right]
 =0\] implies
 \begin{eqnarray}
   (\lambda-\mu)\left<x,e\right>+ \left<x',U^{-1}(a')\right>+\mu sz=0\label{eq:2}
 \end{eqnarray}
 and
 \begin{eqnarray}
 (U-\mu I)(x')+\mu a'z=0.\label{eq:3}
 \end{eqnarray}
  Remark \ref{R2} tells that $\mu \notin S^1$ and hence $U-\mu I$ is invertible as $\sigma(U) \subseteq S^1$.  So by (\ref{eq:3}),\\
 \[x'=-\mu z(U-\mu I)^{-1}(a')\]
 and then (\ref{eq:2}) yields that 
 $\left<x,e\right>=\dfrac{\mu sz+\left<x',U^{-1}(a')\right>}{\mu-\lambda}.$
 If we choose $z \neq 0$ then we always get a non-zero vector $\Big(\left<x,e\right>e+x',z\Big) \in {\text{ker}}(\widehat{T}-\mu I)$ as $s \neq 0$. This yields
   $\sigma(\w{T})=\bigg(\sigma(\w{T})\setminus \{\lambda,\mu\}\bigg)
 \cup \{\lambda,\mu\}=\bigg(\sigma(U)\setminus \{\lambda,\mu\}\bigg)\cup
 \{\lambda,\mu\}=
 \sigma{(U)} \cup
 \{\lambda,\mu\} $ using Lemma \ref{l3}.

\end{proof}

 \subsection{Parabolic isometries}
 Observe that if $\widehat{T} \in \widehat{G}$ is a parabolic isometry of the form (\ref{eq:1}), then $s \neq 0$ in view of Corollary \ref{c6}.
 \begin{lemma}\label{l4}
 If $\widehat{T}$ satisfying (\ref{eq:1}) is a parabolic isometry then $\lambda =\mu$.
 \end{lemma}
 \begin{proof}
 We will prove the result by contradiction. Suppose $\lambda \neq \mu$. As proof of Proposition \ref{m} yields that $\mu$ is an eigenvalue ($s \neq 0$, Corollary \ref{c6}), let the corresponding eigenvector be $w$. Clearly, $w$ cannot be time-like or light-like eigenvector as it would contradict the parabolic nature of $\widehat{T}$. So, $w$ is a space-like vector. Also, $\mathcal{\widehat{A}}(w,w)=\mathcal{\widehat{A}}\Big(\widehat{T}(w),\widehat{T}(w)\Big)=|\mu|^2\mathcal{\widehat{A}}(w,w)$. Eigenvector $w$ being space-like gives $|\mu|=1$  which is a contradiction in view of Remark \ref{R2}.
 \end{proof}
 We shall record the preceding analysis in the following proposition.

 \begin{proposition}\label{p13}
 Spectrum of a parabolic isometry is contained in the unit circle.
 \end{proposition}
 \begin{proof}
 For a parabolic isometry $\widehat{T} \in \widehat{G}$, Proposition \ref{m} and Lemma \ref{l4}   yield $\sigma({\widehat{T}})=\{\lambda\} \cup \sigma(U)$.  Hence $\sigma(\widehat{T}) \subseteq S^1$  by Remark \ref{R2}.
 \end{proof}
 Next we have a spectral characterization of hyperbolic isometries.
 \begin{proposition}\label{p9}
  An isometry is hyperbolic if and only if its spectrum is not contained in the unit circle.
 \end{proposition}
 \begin{proof}
 Theorem \ref{e} (2), Theorem \ref{t1} (2) and Proposition \ref{p13} yield the required result.
 \end{proof}
From now onwards, we will investigate parabolic isometries having singleton spectra.
 \begin{lemma}\label{l11}
 Let $\widehat{T}$ be a parabolic isometry in $\widehat{G}_\infty$ having singleton spectrum (see (\ref{eq:1})). Then  $\widehat{T}=\lambda\left[\begin{array}{ccc}
  1    & \left<\boldsymbol{\cdot},a'\right> &   s\\
    0  &  I & a'\\
    0 & 0 & 1
 \end{array}\right]$ where $s \neq 0$, $|\lambda|=1$ and $\text{Re}\,s=\dfrac{1}{2}\|a'\|^2$.
 \end{lemma}
 \begin{proof}
   Let $\widehat{T}$ satisfying (\ref{eq:1}) have singleton spectrum. Then  $\sigma(\widehat{T})=\{\lambda\}$. By Proposition \ref{m},  $\sigma(U)=\{\lambda\}$ and hence $U= \lambda\,I$. Corollary \ref{c6} tells that $s \neq 0$. Hence the result follows.
   \end{proof}
   \textit{We will denote an isometry  described in the above lemma  by $(\lambda,a',s)$.}

 \begin{proposition}\label{p8}
 The projective maps induced by parabolic isometries in $\widehat{G}_\infty$ having singleton spectra are Heisenberg translations defined as   $\mathcal{\widehat{T}}_{u} : \Sigma \longrightarrow \Sigma$ such that
 \begin{eqnarray*}\widehat{\mathcal{T}}_u(x)=\Big(\left<x,e\right>+\left<u,e\right>+\left<x',u'\right>\Big)e+(x'+u'),
 \end{eqnarray*}
 $u=\left<u,e\right>e+u' \in \partial{\Sigma}$.
 \end{proposition}
 \begin{proof}
Let $\widehat{T}=(\lambda,a',s)$ be a  parabolic isometry in $\widehat{G}_\infty$ having singleton spectrum. Then by Proposition \ref{p10}, its projective map is given by\\
 $\mathcal{\widehat{T}}(x)=\left[\begin{array}{cc}
  1    & \left<\boldsymbol{\cdot},a'\right>\\
    0  & I
 \end{array}\right]\left[\begin{array}{c}
  \left<x,e\right>\\
  x'
 \end{array}\right]+\left[\begin{array}{c}
    s\\
 a'
 \end{array}\right]=\Big(\left<x,e\right>+s+\left<x', a'\right>\Big)e+ x'+ a'= \bigg(\left<x,e\right>+\Big< (se+a'),e\Big>+\left<x', a'\right>\bigg)e+ x'+ a'$.  Let $\mathcal{\widehat{T}}=\mathcal{\widehat{T}}_u$ where $u=se+a' \in \partial{\Sigma}$. Then
 \[\widehat{\mathcal{T}}_u(x)=\Big(\left<x,e\right>+\left<u,e\right>+\left<x',u'\right>\Big)e+(x'+u').\]
 \end{proof}
From now on, we shall refer to an isometry of the form $(\lambda, a',s)$ as Heisenberg translation. The above proposition tells that $u'=a'$ is the component in the horizontal direction. Hence
 \begin{definition}
   A Heisenberg translation of the form $(\lambda,a',s)$ is called  \textbf{\textit{vertical translation}} if $a'=0$  and  \textbf{\textit{non-vertical translation}} if $a'\neq 0$.
 \end{definition}
Recall that for a subspace $M$ of $H \oplus \mathbb{C}$, we use the same notation $M^{\dagger}$ to denote the orthogonal complement of $M$ for the Hermitian form $\widehat{\mathcal{A}}$ as well.
\begin{proposition}\label{p12}
Let $\widehat{T}=(\lambda,a',s)$ be a Heisenberg translation. Then $\widehat{T}$ decomposes $H \oplus \mathbb{C}$ orthogonally into a two or three dimensional subspace  say $K$ and $K^{\dagger}$ such that $\widehat{T}\restriction_{{K}^{\dagger}}= \lambda\,I$. The eigenspace corresponding to $\lambda$ for the operator $\widehat{T}\restriction_{K}$ is generated by a light-like eigenvector and  its minimal polynomial is of maximal degree.
\end{proposition}
\begin{proof}
Let  $\widehat{T}=\lambda\left[\begin{array}{ccc}
  1    & \left<\boldsymbol{\cdot},a'\right> &   s\\
    0  &  I & a'\\
    0 & 0 & 1
 \end{array}\right]$ be a  Heisenberg translation. \\Let $K=span\,\Big\{(e,0),\,(a',0),\,(0,1)\Big\}$. It is easy to see that  $\widehat{T}:  K \longrightarrow K$ is a bijection as  $\widehat{T}(e,0)=\lambda(e,0)$, $\widehat{T}(0,1)=\lambda(se+a',1)$
 and $\widehat{T}(a',0)=\lambda(\|a'\|^2e+ a',0)$. In view of Lemma \ref{parker},  $K$ is a non-degenerate subspace as it contains two light-like vectors.  We have $H \oplus \mathbb{C}=K \oplus K^{\dagger}$ by Lemma \ref{direct sum}.  Observe that $K^{\dagger}=\{(x',0)\,\,:\,\,\left<x',a'\right>=0\}$ and ${\widehat{T}\restriction_{K^{\dagger}}}=\lambda\,I$. Let $\widehat{T}=\widehat{T}_1 \oplus \widehat{T}_2$ where $\widehat{T}_1=\widehat{T}\restriction_K$ and $\widehat{T}_2=\widehat{T}\restriction_{K^{\dagger}}$.
   Next we see that  $ker(\widehat{T}_1- \lambda\,I)=\left<(e,0)\right>$. Observe that\\ $(\widehat{T}- \lambda\,I)(x,z)=\lambda\left[\begin{array}{ccc}
  0    & \left<\boldsymbol{\cdot},a'\right> &   s\\
    0  & 0 & a'\\
    0 & 0 & 0
 \end{array}\right]\left[\begin{array}{c}
 \left<x,e\right>\\
 x'\\
 z
 \end{array}\right]=\left[\begin{array}{c}
 0\\
 0\\
 0
 \end{array}\right]$\\
 implies  $\left<x',a'\right>+sz=0$ and $a'z=0$
  giving $z=0$ and $\left<x',a'\right>=0$. Hence $ker\,(\widehat{T}- \lambda\,I)=\left<(e,0)\right> \oplus K^{\dagger}$ and thus $ker\,{(\widehat{T}_1- \lambda\,I)}=\left<(e,0)\right>$. Consider
  ${(\widehat{T}_1- \lambda\,I)}^2(0,1)=\lambda^2\left[\begin{array}{ccc}
  0    & \left<\boldsymbol{\cdot},a'\right> &   s\\
    0  & 0 & a'\\
    0 & 0 & 0
 \end{array}\right]\left[\begin{array}{ccc}
  0    & \left<\boldsymbol{\cdot},a'\right> &   s\\
    0  & 0 & a'\\
    0 & 0 & 0
 \end{array}\right]\left[\begin{array}{c}
 0\\
 0\\
 1
 \end{array}\right]=\\
 \lambda^2 \left[\begin{array}{ccc}
  0    & \left<\boldsymbol{\cdot},a'\right> &   s\\
    0  & 0 & a'\\
    0 & 0 & 0
 \end{array}\right]\left[\begin{array}{c}
 s\\
 a'\\
 0
 \end{array}\right]=\lambda^2\left[\begin{array}{c}
 \|a'\|^2\\
 0\\
 0
 \end{array}\right] \not\equiv 0$ for $a' \neq 0$. Hence for a non-vertical translation, minimal polynomial of $\widehat{T}_1$ is $(x-\lambda)^3$ as $K$ is a three dimensional subspace. Observe that if $\widehat{T}$ is a vertical translation, i.e. $a'=0$ then $K=\left<(e,0)\right> \oplus \mathbb{C}$ is a two dimensional subspace and hence minimal polynomial is $(x-\lambda)^2$.
\end{proof}

\begin{proposition}\label{p11}
(1) All non-vertical Heisenberg translations having same spectra fall in a single conjugacy class. \\
(2) All vertical Heisenberg translations having same spectra  get dispersed into two conjugacy classes.
 \end{proposition}
 \begin{proof}
 Observe that any isometry in $\widehat{G}$ conjugating one Heisenberg translation into the other has to fix $\infty$.\\
(1)   We will see that any two non-vertical Heisenberg translations of the form $(\lambda,a',s)$ and $(\lambda,b', t)$ are conjugate to each other, i.e. to show the existence of an isometry $\widehat{R}=\left[\begin{array}{ccc}
  \lambda'    & \left<\boldsymbol{\cdot},U^{-1}(c')\right> &   \mu' s_1\\
    0  & U & \mu' c'\\
    0 & 0 & \mu'
 \end{array}\right] \in \widehat{G}_\infty$ such that $\widehat{R}(\lambda,a',s)\widehat{R}^{-1}=(\lambda,b', t)$. This is same thing as showing $\widehat{R}(1,a',s)\widehat{R}^{-1}=(1,b', t)$.  We have \\
 $\widehat{R}(1,a',s)=\left[\begin{array}{ccc}
  \lambda'    & \left<\boldsymbol{\cdot},U^{-1}(c')\right> &   \mu' s_1\\
    0  & U & \mu' c'\\
    0 & 0 & \mu'
 \end{array}\right]\left[\begin{array}{ccc}
  1    & \left<\boldsymbol{\cdot},a'\right> &   s\\
    0  & I & a'\\
    0 & 0 & 1
 \end{array}\right]$\\
 =$\left[\begin{array}{ccc}
  \lambda'    & \lambda'\left<\boldsymbol{\cdot},a'\right>+\left<\boldsymbol{\cdot},U^{-1}(c')\right> &   \lambda'\,s+\left<U(a'),c'\right>+\mu'\, s_1\\
    0  & U & U(a')+\mu'\,c'\\
    0 & 0 & \mu'
 \end{array}\right]$\\
 and $(1,b', t)\widehat{R}=\left[\begin{array}{ccc}
  1    & \left<\boldsymbol{\cdot},b'\right> &   t\\
    0  & I & b'\\
    0 & 0 & 1
 \end{array}\right]\left[\begin{array}{ccc}
  \lambda'    & \left<\boldsymbol{\cdot},U^{-1}(c')\right> &   \mu' \,s_1\\
    0  & U & \mu' c'\\
    0 & 0 & \mu'
 \end{array}\right]$\\
 =$\left[\begin{array}{ccc}
  \lambda'    & \left<\boldsymbol{\cdot},U^{-1}(c')\right>+\left<\boldsymbol{\cdot},U^{-1}(b')\right> &   \mu'\,s_1+\left<\mu'\,c',b'\right>+\mu'\,t\\
    0  & U & \mu'\,c'+\mu'\,b'\\
    0 & 0 & \mu'
 \end{array}\right]$.\\
 From the above computations we see that $\widehat{R}(1, a',s)=(1, b', t)\widehat{R}$ if and only if
 \begin{align}
 U(a')&=\mu'\,b'\,\,\,\,\text{and}\label{eq:12}\\
 \lambda'\,s+\left<U(a'),c'\right>&=\left<\mu'\,c',b'\right>+\mu'\,t \label{eq:13}.
 \end{align}
 Choose $\mu'=\dfrac{\|a'\|}{\|b'\|}$ and a unitary operator $U$ which takes an orthonormal basis in $\left<a'\right>^{\perp}$ to an orthonormal basis in $\left<b'\right>^{\perp}$. We need to show the existence of $c'$ satisfying $\lambda'\,s+\mu'\left<
 b',c'\right>=\mu'\,\left<c',b'\right>+\mu'\,t$, i.e. $|\lambda'|^2\,s-t=\left<c',b'\right>-\left<b',c'\right>=2i\,\text{Im}\,(\left<c',b'\right>)$. By putting $\lambda'=\dfrac{\|b'\|}{\|a'\|}$ and using $\text{Re}\,s=\dfrac{1}{2}\|a'\|^2$ and $\text{Re}\,t=\dfrac{1}{2}\|b'\|^2$,  we are able to find a suitable $c'$.

 (2) Equations (\ref{eq:12}) and (\ref{eq:13}) tell that two vertical Heisenberg translations $(\lambda,0,s)$ and $(\lambda,0,t)$  are conjugate to each other if and only if there exists $\lambda'$ satisfying $|\lambda'|^2\,s=t$, i.e. $|\lambda'|^2=\dfrac{t}{s}=\dfrac{\text{Im}\,t}{\text{Im}\,s}$ as $\text{Re}\,s=\text{Re}\,t=0$. Hence all vertical translations of the form $(\lambda,0,s)$ with $\text{Im}\,s>0$ fall in one conjugacy class and those with $\text{Im}\,s<0$  fall in the other conjugacy class.
  \\
 It is easy to see that $(\lambda,a',s)$ cannot be conjugated to $(\lambda, 0, t)$ as (\ref{eq:12}) tells $b'=0$ if and only if $a'=0$.
 \end{proof}

 \textbf{Proof of Theorem \ref{o}.} (1) If $\widehat{S}$ is a parabolic isometry having singleton spectrum  then by Lemma \ref{bitransitive}, $\widehat{S}$ can be conjugated to a parabolic isometry in $\widehat{G}_\infty$ having singleton spectrum, i.e. a Heisenberg translation.  \\
 (2) Follows from Proposition \ref{p12} in view of (1).\\
 (3) By (2) and (1),  a parabolic isometry having singleton  spectrum and degree of restricted minimal polynomial as three (resp. two)  is conjugate to a non-vertical (resp. vertical) Heisenberg translation. Now Proposition \ref{p11} asserts the claim.\qed
 \section{Centralizers of isometries}

Centralizer of an element $T \in G$ will be denoted by $Z(T)$. Recall that for any two bounded linear operators  $T$ and $S$ on a Hilbert space, $T \longleftrightarrow S$ means $T$  and $S$ commute with each other. Some  results of this section are motivated from \cite{WCKG} and \cite{g}.
 Using the terminology of \cite{WCKG} we have the following.
  \begin{definition}\label{d2}
   For an isometry $T \in G$, an eigenvalue of $T$ is called \textbf{\textit{time-like eigenvalue}} if its corresponding eigenspace is time-like.
 We call an elliptic isometry \textbf{\textit{boundary  elliptic}} if the geometric multiplicity of its time-like eigenvalue is atleast two, otherwise it is called  \textbf{\textit{regular  elliptic}}.
  \end{definition}
\begin{lemma}\label{l8}
Let $T,\,S \in G$. If $T \longleftrightarrow S$ and $E$ is an eigenspace of $T$ then $S(E)=E$.
\end{lemma}
\begin{proof}
The fact that $S^{-1}$ also commutes with $T$ proves the claim.
\end{proof}
 \begin{corollary}\label{c4}
  Regular elliptic isometries commute only with elliptic isometries.
  \end{corollary}

      \begin{lemma}\label{LEMMA}
          Let $M$ be a non-degenerate subspace of $H \oplus \mathbb{C}$ such that $H \oplus \mathbb{C}=M \oplus M^{\dagger}$. Then for any $R \in G$,  $H \oplus \mathbb{C}=R(M) \oplus {R(M)}^{\dagger}$.

  \end{lemma}
  \begin{proof}
   Clearly, $R(M) \oplus R(M)^{\dagger} \subseteq H \oplus \mathbb{C}$. We will show that $H \oplus \mathbb{C} \subseteq R(M) \oplus R(M)^{\dagger}$. As $M$ is non-degenerate, $R(M) \cap {R(M)}^{\dagger} =\{0\}$. Let $y \in H \oplus \mathbb{C}$ and $R(m) \in R(M)$ for some $m \in M$. We will see that $y -R(m) \in R(M)^{\dagger}$. For $x \in M$, consider $\mathcal{A}(y-R(m),R(x))=\mathcal{A}(y,R(x))-\mathcal{A}(R(m),R(x))=\mathcal{A}(R^{-1}(y),x)-\mathcal{A}(m,x)=\mathcal{A}(R^{-1}(y)-m,x)=0$ as $H \oplus \mathbb{C}=M \oplus M^{\dagger}$, $x \in M$, $R^{-1}(y) \in H \oplus \mathbb{C}$ and thus $R^{-1}(y)-m \in M^{\dagger}$.
   \end{proof}
 Recall that for  a closed subspace $M \subseteq H \oplus \mathbb{C}$,   $G_{M}$ denotes the group of all bijective bounded linear mappings on $M$ leaving the form $\mathcal{A}\restriction_{M}$ invariant.

  \textbf{Proof of Theorem \ref{p1}.}  First we will show that $T$ decomposes $H \oplus \mathbb{C}$ into $M$ and $M^{\dagger}$. Let $T=R\wt{U}R^{-1}$ where $\wt{U} \in G$ is of the form $e^{i \theta}\left[\begin{array}{cc}
  U & 0\\
  0 & 1\\
  \end{array}\right]$. Observe that $e^{i \theta}$ is the only time-like eigenvalue of $T$ and $\wt{U}$ both. As $M$ is a time-like eigenspace of $T$ (corresponding to eigenvalue $e^{i \theta}$),  $R^{-1}(M)$ is an eigenspace of $\wt{U}$ corresponding to eigenvalue $e^{i \theta}$ and hence this eigenspace contains  $\mathbb{C}$. So $R^{-1}(M)$ is a closed subspace of $H \oplus \mathbb{C}$ . As $R^{-1}(M)^{\dagger}=R^{-1}(M)^{\perp}$ ($R^{-1}(M)^{\perp} \subseteq H$), $H \oplus \mathbb{C}=R^{-1}(M) \oplus (R^{-1}(M))^{\dagger}$. Lemma \ref{LEMMA} tells $H \oplus \mathbb{C}=M \oplus M^{\dagger}$. Also $T(M)=M$ gives that $T(M^{\dagger})=M^{\dagger}$ (Lemma \ref{b}). Hence $T$ decomposes $H \oplus \mathbb{C}$ into $M$ and $M^{\dagger}$. Let $S \in Z(T)$. Lemma \ref{l8} tells that $S(M)=M$ and hence $S(M^{\dagger})=M^{\dagger}$ by Lemma \ref{b}.   Now we show that $M^{\dagger}$ is a closed subspace of $H \oplus \mathbb{C}$. As $\mathcal{A}\restriction_{{M}^{\dagger}}$ is positive definite by Lemma \ref{parker}, it is  an inner product space. Since $R^{-1}(M^{\dagger})=R^{-1}(M)^{\dagger}=R^{-1}(M)^{\perp}$, $R^{-1}:M^{\dagger} \longrightarrow (R^{-1}(M))^{\perp}$ is a surjective linear isometry.  This gives that $M^{\dagger}$  is a Hilbert space as  $(R^{-1}(M))^{\perp}$ is complete. The above analysis yields  $Z(T)=Z(T\restriction_{M}) \times Z(T\restriction_{{M}^{\dagger}})$ where  $Z(T\restriction_{{M}^{\dagger}})\subseteq \mathcal{U}(M^{\dagger})$. As $T\restriction_{M}$ is a scalar multiple of identity, $Z(T\restriction_{M})=G_{M}$.  \qed

 \subsection{Centralizers of hyperbolic isometries} In this section we provide proof of Theorem \ref{p3}.\\
 \textbf{Proof of Theorem \ref{p3}.}  As $M$ is a non-degenerate subspace by Lemma \ref{parker}, $T$ decomposes $H \oplus \mathbb{C}$ into $M$ and $M^{\dagger}$ (Lemma \ref{direct sum} and Lemma \ref{b}). Let $S \in Z(T)$. Lemma \ref{l8} tells that $S$ leaves each one-dimensional eigenspace invariant. Hence $S(M)=M$ and $S(M^{\dagger})=M^{\dagger}$ (Lemma \ref{b}).    The same argument as in the proof of Theorem \ref{t1} (3) shows that $M^{\dagger}$ is a Hilbert space. So we have $Z(T)=Z(T\restriction_{M})   \times Z(T\restriction_{M^{\dagger}})  $ where $Z(T\restriction_{M}) \subseteq G_M$ and  $Z(T\restriction_{M^{\dagger}}) \subseteq \mathcal{U}(M^{\dagger})$. We will show that $Z(T\restriction_{M})$ gets identified with $S^1 \times \mathbb{R}$. Let $M=\text{span}\{(y,1),(z,1)\}$ where $(y,1)$ and $(z,1)$ are two light-like eigenvectors of $T$. Let $S'=S\restriction_M \in Z(T\restriction_{M})$. Then $S'(y,1)=\alpha(y,1)$ and $S'(z,1)=\beta(z,1)$ for some $\alpha,\,\beta \in \mathbb{C}$. This implies that $S'$ is completely determined by its eigenvalues.  As $S' \in G_M$ too,  $\mathcal{A}((y,1),(z,1))=\mathcal{A}(S'(y,1),S'(z,1))$, i.e. $\left<y,z\right>-1=\alpha \overline{\beta}(\left<y,z\right>-1)$ if and only if $\alpha \overline{\beta}=1$ by  Lemma \ref{parker}. Now there are two possibilities either $\alpha=\beta$ or $\alpha \neq \beta$. $\alpha=\beta$ if and only if $\alpha \in S^1$ by Remark \ref{R2}.  $\alpha \neq \beta$ if and only if $S'$ is a hyperbolic isometry  and hence $\alpha$ and $\beta$ are of the form $re^{i \theta}$ and $r^{-1}e^{i \theta}$ by Theorem \ref{t1} (2). This tells that for $re^{i \theta} \in S^1 \times \mathbb{R}$, there corresponds a unique $S' \in Z(T\restriction_{M})$ such that $S'(y,1)=re^{i \theta}(y,1)$ and $S'(z,1)=r^{-1}e^{i \theta}(z,1)$. Hence every element of $S^1 \times \mathbb{R}$ uniquely determines an element of $Z(T\restriction_{M})$. \qed
 
\medskip 

Following is the description of centralizers in the group of unitary operators defined on a separable Hilbert  space $H$. We note it here without proof. \\
Using the multiplicity theory for normal operators on a separable Hilbert space $H$ \cite[Chapter IX, section 10]{JBC}, we will compute centralizer of a unitary operator in $\mathcal{U}(H)$. 
 
 The integral representation of normal operators with respect to spectral measure yields that if a normal operator $N$ is star-cyclic \cite[Chapter IX, Definition 3.1]{JBC}, then it is unitarily equivalent to the multiplication operator ( multiplication by identity) defined on $L^2(\sigma(N),\,\mu)$ where $\mu$ is a scalar valued spectral measure. Now, decomposing $H$ into countable orthogonal pieces such that the normal operator restricted to each piece is star-cyclic gives that $N$ is unitarily equivalent to the direct sum of multiplication operators. In precise form, we have the following.
 
 \begin{theorem}\cite[Chaper IX, Theorem 10.20]{JBC}
 If $N$ is a normal operator on $H$, then there are mutually singular measures $\mu_{\infty}$, $\mu_{1}$, $\mu_{2},\,\cdot \cdot \cdot$ and an isomorphism $U: H \longrightarrow L^2(\mu_{\infty};\, H_{\infty}) \oplus L^2(\mu_{1}) \oplus L^2(\mu_{2};\,H_{2}) \oplus \,\cdot \cdot \cdot$ such that $UNU^{-1}=N_{\infty} \oplus N_{1} \oplus N_{2} \oplus\,\cdot \cdot \cdot$ where $H_n$ is an n-dimensional Hilbert space, $L^2(\mu_n;\,H_n)$ is the space of square integrable $H_n$ valued functions and $N_n$ is multiplication by $z$ on $L^2(\mu_n;\,H_n)$. 
 \end{theorem}
 Also, if $N$ is multiplication by $z$ on $L^2(\mu;\,H_n)$ and $M_{\phi}$ is multiplication by $\phi$, then $$\{N\}'=\left\{M_{\phi}\,\,:\,\,\phi \in L^{\infty}(\mu;\,B(H_n))\right\},$$
 see \cite[Chapter IX, Corollary 6.9]{JBC}.
 This gives $$\{N_{\infty}\oplus N_1 \oplus N_2\oplus\, \cdot \cdot \cdot\}'=L^{\infty}(\mu_{\infty};\,B(H_{\infty})) \oplus L^{\infty}(\mu_1)\oplus L^{\infty}(\mu_2;\,B(H_2))\oplus\,\cdot \cdot \cdot.$$

 As a corollary to the above analysis,  we have the following result.
 \begin{proposition} \label{k}
 If $V$ is a unitary operator on a separable Hilbert space $H$, then
 $$Z(V)=U^{-1}Z(V_{\infty}\oplus V_1\oplus V_2 \oplus...)U$$
 where $U$ is as in the preceding theorem and $Z(V_{\infty}\oplus V_1\oplus V_2 \oplus...)=U(H_{\infty})\oplus U(H_1)\oplus U(H_2)\oplus...$, $U(H_n)$ is the group of unitary elements in $L^{\infty}(\mu;\,B(H_n)).$
 
\end{proposition}
For the non-separable case, one may refer to \cite{HAL}.

  \subsection{Centralizers of  Heisenberg translations}
  \begin{definition}\label{isotropic}
  Two Heisenberg translations $(\lambda,a',s)$ and $(\lambda',b',t)$ are said to be \textbf{\textit{isotropic}} if $\left<b',a'\right> \in \mathbb{R}$, i.e their translations in the horizontal direction have real product.
  \end{definition}
  $\widehat{T}=\lambda\left[\begin{array}{ccc}
  1    & \left<\boldsymbol{\cdot},a'\right> &   s\\
    0  &  I & a'\\
    0 & 0 & 1
     \end{array}\right]=(\lambda,a',s)
     $
     will denote  a  Heisenberg translation,  where $s \neq 0$, $|\lambda|=1$ and $\text{Re}\,s=\dfrac{1}{2}{\|a'\|}^{2},$ throughout the subsection.
 \begin{lemma}\label{l10}
 If $\widehat{S} \longleftrightarrow \widehat{T}$ then $\widehat{S}$ is either a boundary elliptic or parabolic isometry in $\widehat{G}_{\infty}$.
 \end{lemma}
 \begin{proof}
 Corollary \ref{c4} and Corollary \ref{c5} tell that  parabolic isometries commute either with boundary elliptics or parabolics. Also $\w{T}$ fixes only $\infty$ and $\w{S} \longleftrightarrow \w{T}$. This implies $\w{S}$ fixes $\infty$ and hence $\widehat{S}$ lies in $\widehat{G}_{\infty}$.
\end{proof}

\begin{proposition}
Two Heisenberg translations  commute with each other if and only if they are isotropic.
\end{proposition}
 \begin{proof}
Let $(\lambda,a',s)$ and $(\lambda',b',t)$ be two Heisenberg translations.   We have
\begin{align*}
 (\lambda,a',s)(\lambda',b',t)&=\lambda \lambda'\left[\begin{array}{ccc}
  1   & \left<\boldsymbol{\cdot},a'\right> &   s\\
    0  &  I & a'\\
    0 & 0 & 1
 \end{array}\right]\left[\begin{array}{ccc}
  1    & \left<\boldsymbol{\cdot},b'\right> &   t\\
    0  & I & b'\\
    0 & 0 & 1
 \end{array}\right]\\
 &=\lambda \lambda'\left[\begin{array}{ccc}
      1 & \left<\boldsymbol{\cdot},b'\right>+\left<\boldsymbol{\cdot},a'\right> &   t+\left<b',a'\right>+s\\
    0  & I & b'+a'\\
    0 & 0 & 1
 \end{array}\right]\,\,\text{and}\\
 (\lambda,b',t)(\lambda,a',s)&=\lambda' \lambda\left[\begin{array}{ccc}
  1    & \left<\boldsymbol{\cdot},b'\right> &   t\\
    0  & I & b'\\
    0 & 0 & 1
 \end{array}\right]\left[\begin{array}{ccc}
  1   & \left<\boldsymbol{\cdot},a'\right> &   s\\
    0  &  I & a'\\
    0 & 0 & 1
 \end{array}\right]\\
 &=\lambda' \lambda\left[\begin{array}{ccc}
  1    & \left<\boldsymbol{\cdot},a'\right>+ \left<\boldsymbol{\cdot},b'\right> &   s +\left<a',b'\right>+t\\
    0  & I & a'+b'\\
    0 & 0 & 1
 \end{array}\right].
 \end{align*}
 This gives that $(\lambda,a',s) \longleftrightarrow (\lambda,b',t)$ if and only if they are isotropic.
 \end{proof}

\textbf{Proof of Theorem \ref{t3}.}
  Let $\widehat{S} \in \widehat{G}$ be such that $\widehat{S} \longleftrightarrow \widehat{T}$. Lemma \ref{l10} tells that $\widehat{S} \in \widehat{G}_\infty$ and it is either boundary elliptic or parabolic. This implies $\sigma(\widehat{S}) \subseteq S^1$. As $\widehat{G}_\infty$ contains only boundary elliptic, hyperbolic and parabolic isometries fixing infinity, its general isometry  $\left[\begin{array}{ccc}
  \lambda'    & \left<\boldsymbol{\cdot},U^{-1}(b')\right> &   \mu'\,t\\
    0  & U & \mu'\,b'\\
    0 & 0 & \mu'
 \end{array}\right]$ is hyperbolic if and only if $\lambda' \neq \mu'$ in view of Proposition \ref{p9}, Proposition \ref{m} and Remark \ref{R2}.  Hence a boundary elliptic or parabolic  isometry in $\widehat{G}_\infty$ is of the form  $\left[\begin{array}{ccc}
  \lambda'    & \left<\boldsymbol{\cdot},U^{-1}(b')\right> &   \lambda'\,t\\
    0  & U & \lambda'\,b'\\
    0 & 0 & \lambda'
 \end{array}\right]=\widehat{S}$ say.
 \begin{enumerate}
     \item For $\widehat{T}=\lambda\left[\begin{array}{ccc}
  1    & 0 &   s\\
    0  &  I & 0\\
    0 & 0 & 1
 \end{array}\right]$,
 \begin{align*}
   \widehat{S}\widehat{T}&=\lambda\left[\begin{array}{ccc}
  \lambda'    & \left<\boldsymbol{\cdot},U^{-1}(b')\right> &   \lambda'\,t\\
    0  & U & \lambda'\,b'\\
    0 & 0 & \lambda'
 \end{array}\right]\left[\begin{array}{ccc}
  1    & 0 &   s\\
    0  &  I & 0\\
    0 & 0 & 1
 \end{array}\right]\\
 &=\lambda\left[\begin{array}{ccc}
  \lambda'    & \left<\boldsymbol{\cdot},U^{-1}(b')\right> &   \lambda'\,s+\lambda'\,t\\
    0  & U & \lambda'\,b'\\
    0 & 0 & \lambda'
 \end{array}\right]\\
  &= \widehat{T}\widehat{S}\\
  &=\lambda\left[\begin{array}{ccc}
  1    & 0 &   s\\
    0  &  I & 0\\
    0 & 0 & 1
 \end{array}\right]\left[\begin{array}{ccc}
  \lambda'    & \left<\boldsymbol{\cdot},U^{-1}(b')\right> &   \lambda'\,t\\
    0  & U & \lambda'\,b'\\
    0 & 0 & \lambda'
 \end{array}\right].
 \end{align*}
  Hence the result.
 \item For $\widehat{T}=\lambda\left[\begin{array}{ccc}
  1    & \left<\boldsymbol{\cdot},a'\right> &   s\\
    0  &  I & a'\\
    0 & 0 & 1
 \end{array}\right]$, $a' \neq 0$,
 \begin{align*}
 \widehat{S}\widehat{T}&=\lambda\left[\begin{array}{ccc}
  \lambda'    & \left<\boldsymbol{\cdot},U^{-1}(b')\right> &   \lambda'\,t\\
    0  & U & \lambda'\,b'\\
    0 & 0 & \lambda'
 \end{array}\right]\left[\begin{array}{ccc}
  1    & \left<\boldsymbol{\cdot},a'\right> &   s\\
    0  &  I & a'\\
    0 & 0 & 1
 \end{array}\right]\\
 &=\lambda\left[\begin{array}{ccc}
  \lambda'    & \lambda'\left<\boldsymbol{\cdot},a'\right>+ \left<\boldsymbol{\cdot},U^{-1}(b')\right> &   \lambda'\,s+ \left<U(a'),b'\right>+\lambda'\,t\\
    0  & U & U(a')+\lambda'\,b'\\
    0 & 0 & \lambda'
 \end{array}\right]
 \end{align*}
  and
  \begin{align*}
  \widehat{T}\widehat{S}&=\lambda\left[\begin{array}{ccc}
  1    & \left<\boldsymbol{\cdot},a'\right> &   s\\
    0  &  I & a'\\
    0 & 0 & 1
 \end{array}\right]\left[\begin{array}{ccc}
  \lambda'    & \left<\boldsymbol{\cdot},U^{-1}(b')\right> &   \lambda'\,t\\
    0  & U & \lambda'\,b'\\
    0 & 0 & \lambda'
 \end{array}\right]\\
 &=\lambda\left[\begin{array}{ccc}
  \lambda'    &  \left<\boldsymbol{\cdot},U^{-1}(b')\right>+\left<\boldsymbol{\cdot},U^{-1}(a')\right> &    \lambda'\,t+\left<\lambda'\,b',a'\right>+\lambda'\,s\\
    0  & U & \lambda'\,b'+\lambda'\,a'\\
    0 & 0 & \lambda'
 \end{array}\right].
  \end{align*}  \end{enumerate} 

  This tells that $\widehat{S} \longleftrightarrow \widehat{T}$ if and only if $U(a')=\lambda'\,a'$ and $\left<a',b'\right> \in \mathbb{R}$. \qed

 \bibliographystyle{amsplain}

\end{document}